\documentclass{amsart}
 \newtheorem{thm}{Theorem}[section]
 \newtheorem{cor}[thm]{Corollary}
 \newtheorem{lem}[thm]{Lemma}
 \newtheorem{prop}[thm]{Proposition}
 \newtheorem{cjt}{Conjecture}
 \theoremstyle{definition}
 \newtheorem{defn}[thm]{Definition}
 \newtheorem{ex}[thm]{Example}
 \theoremstyle{remark}
 \newtheorem{rem}[thm]{Remark}
 \newtheorem{notation}[thm]{Notation}
 \numberwithin{equation}{section}

\usepackage[hidelinks]{hyperref}
\usepackage{tikz}
\usepackage{enumitem}
\usepackage{pgf}
\usepackage{amssymb}
\usepackage{nicefrac}
\usetikzlibrary{shapes,matrix,calc,arrows,decorations.markings,knots,patterns,intersections,cd}

\newcommand\PP{\mathbb{P}}
\newcommand\ZZ{\mathbb{Z}}
\newcommand{\CC}{\mathbb{C}}
\newcommand{\cC}{\mathcal{C}}
\newcommand{\cD}{\mathcal{D}}
\newcommand{\Q}{\mathbb{Q}}
\newcommand{\s}{s}
\newcommand{\dd}{\delta}
\DeclareMathOperator{\rk}{Rank}
\DeclareMathOperator{\coker}{coker}
\DeclareMathOperator{\sing}{Sing}
\DeclareMathOperator{\lcm}{lcm}
\DeclareMathOperator{\mult}{mult}
\DeclareMathOperator{\supp}{V}
\DeclareMathOperator{\Jac}{Jac}

\newcommand{\orb}{{\text{\rm orb}}}

\begin{document}

\title[Cremona transformations of weighted projective planes, Zariski pairs, ...]
{Cremona transformations of weighted projective planes, Zariski pairs, and rational cuspidal curves}

\author[E.~Artal]{Enrique Artal Bartolo}
\address{Departamento de Matem\'{a}ticas, IUMA \\
Universidad de Zaragoza \\
C.~Pedro Cerbuna 12, 50009, Zaragoza, Spain}
\urladdr{http://riemann.unizar.es/\~{}artal}
\email{artal@unizar.es}

\author[J.I.~Cogolludo]{Jos{\'e} Ignacio Cogolludo-Agust{\'i}n}
\address{Departamento de Matem\'aticas, IUMA \\ 
Universidad de Za\-ra\-go\-za \\ 
C.~Pedro Cerbuna 12 \\ 
50009 Zaragoza, Spain} 
\urladdr{http://riemann.unizar.es/\~{}jicogo}
\email{jicogo@unizar.es} 

\author[J.~Mart\'{\i}n-Morales]{Jorge Mart\'{\i}n-Morales}
\address{Centro Universitario de la Defensa, IUMA \\
Academia General Militar \\ 
Ctra.~de Huesca s/n. \\ 
50090, Zaragoza, Spain}
\urladdr{http://cud.unizar.es/martin}
\email{jorge@unizar.es}

\thanks{Partially supported by MTM2016-76868-C2-2-P and 
Gobierno de Arag{\'o}n (Grupo de referencia ``{\'A}lgebra y Geometr{\'i}a'')
cofunded by Feder 2014-2020 ``Construyendo Europa desde Arag\'on''.
The third author is also partially supported by FQM-333 ``Junta de Andaluc{\'\i}a''.}
\date{\today}

\dedicatory{To Andr{\'a}s N{\'e}methi, source of inspiration in singularity theory}

\begin{abstract}
In this work, we study a family of Cremona transformations of weighted projective planes which
generalize the standard Cremona transformation of the projective plane. Starting from special
plane projective curves we construct families of curves in weighted projective planes with special
properties. We explain how to compute the fundamental groups of their complements, using the
blow-up-down decompositions of the Cremona transformations, we find examples of Zariski pairs
in weighted projective planes (distinguished by the Alexander polynomial). As another application 
of this machinery we study a family of singularities called weighted L\^{e}-Yomdin, which provide
infinite examples of surface singularities with a rational homology sphere link. To end this 
paper we also study a family of surface singularities generalizing Brieskorn-Pham singularities
in a different direction. This family contains infinitely many new examples of integral homology 
sphere links, answering a question by N\'emethi. 
\end{abstract}

\maketitle


\section*{Introduction}
This paper deals with curves in surfaces with normal singularities and the interplay between their topological and algebraic properties.

In this direction we provide a family of examples of curves in weighted projective planes using 
a generalization of the classical Cremona transformations. This allows us to construct infinitely 
many pairs of curves in weighted projective planes defining linearly equivalent divisors and the same 
local type of singularities, whose embeddings are not homeomorphic. Moreover, whose complements have 
non-isomorphic fundamental groups. This is known in the literature as Zariski pairs when referred 
to plane projective curves~\cite{Artal94} since Zariski provided the first example of such a phenomenon
in~\cite{Zariski-irregularity}. The curves are obtained from a smooth cubic and three tangent lines
via a weighted Cremona transformation in~\S\ref{sec:cremona}. These groups are distinguished using two 
different techniques. In~\S\ref{sec:fund-groups} a topological approach is given by obtaining 
presentations of the groups. These presentations, which in general are complicated to calculate, can 
be derived from those of the original curve after Cremona transformations in a very explicit geometric way. 
To complete this example, we also present a more algebraic approach via cyclic coverings as was
originally used by Zariski and later developed by Steenbrink~\cite[Lemma 3.14]{steenbrink-mhs},
Libgober~\cite{Libgober-alexander}, Esnault-Viehweg~\cite{Esnault-Viehweg82}, 
Vaqui\'e~\cite{Vaquie-irregulatity}, and the first author~\cite{Artal94}. 
Our method uses a generalization of~\cite{Esnault-Viehweg82} given
in~\cite{ACM19}, see~\S\ref{sec:EsnaultViehweg}. 
Section~\ref{sec:ratcusp} is devoted to developing some methods to construct rational cuspidal 
curves in weighted projective planes which will be useful in the later sections.

The second part of the paper focuses on local properties of surface singularities. Our main goal
is to provide examples of surface germs whose link is a rational (or even integral) homology sphere.
A source of examples is given by superisolated singularities. In section~\ref{sec:LY} we introduce
the determinant of a surface singularity as the absolute value of the determinant of the intersection 
matrix of a resolution. This invariant of the surface singularity can also be calculated using a
partial resolution, as shown in~\S\ref{sec:determinant}. Note that a surface singularity has a rational 
homology sphere link if and only if the dual graph of a (partial) resolution is a tree whose vertices 
are rational curves. Moreover, a rational homology sphere link is integral if the determinant of the 
singularity is one. We use this criterion to study weighted L\^{e}-Yomdin singularities and to describe 
infinite families with rational and integral homology sphere links.

In particular, following the ideas in~\cite{artal-formeJordan,JM-semistable}, one can use the 
Zariski pairs obtained in section~\ref{sec:zp} to construct weighted L\^{e}-Yomdin singularities 
having the same Alexander polynomials, the same abstract topology, but different embedded topology.
It would be hopeless to compute the Jordan form of the complex monodromy (the actual invariant that distinguishes
the embedded topology) without the use of the techniques in this paper.

The last part is devoted to solving two problems on surface singularities with a rational sphere link.
Namely, in \S\ref{ex:bp} we study Brieskorn-Pham surface singularities 
$\{x^a+y^b+z^c=0\}\subset \mathbb{C}^3$ as a special case of weighted L\^{e}-Yomdin. We illustrate how to recover
classical results in a simple way, namely to characterize which ones have a rational sphere link and show that the only 
integral homology spheres occur in the classical case, that is, whenever $(a,b,c)$ are pairwise coprime.
Besides Brieskorn-Pham singularities, more examples are provided in \S\ref{sec:RHSCremona}
using weighted Cremona transformations and Kummer covers. 

Andr\'as N\'emethi asked us if it was possible to find singularities with integral homology sphere links
in the realm of weighted L\^{e}-Yomdin singularities. The only ones we found are the already known Brieskorn-Pham 
singularities. As an alternative, in \S\ref{sec:newexamples}, a new family of surface singularities is presented 
following~\cite{MVV19}. We give conditions for this family to have a rational homology sphere link.
Moreover, this family provides infinitely many new examples of integral homology sphere links
which may answer the question by Andr\'as N\'emethi in the affirmative.

\section{Quotient singularities and weighted Cremona transformations}
\label{sec:intro}

The main objects of this work will be weighted projective planes (and lines) and quotient singularities. 
A quotient singularity is a normal space which is locally isomorphic to $(X,0)$ where $X$ is the quotient 
of $\CC^n$ by the action of a cyclic group $\mu_m\subset\CC^*$ given by
\[
\zeta\cdot(x_1,\dots,x_n)=(\zeta^{a_1} x_1,\dots,\zeta^{a_n} x_n),\qquad \zeta^m=1,(x_1,\dots,x_n)\in\CC^n.
\]
If $\gcd(m,a_1,\dots,a_n)=1$, the action is faithful. We denote this singularity by
$\frac{1}{m}(a_1,\dots,a_n)$. There are some trivial equivalences of quotient singularities such as 
$\frac{1}{m}(a_1,\dots,a_n)=\frac{1}{m}(d a_1,\dots,d a_n)$ if $\gcd(m,d)=1$. 
A less obvious one is given by
\[
\frac{1}{m}(a_1,\dots,a_n)\cong\frac{d}{m}\left(a_1,\frac{a_2}{d},\dots,\frac{a_n}{d}\right)
\text{ if }d=\gcd(m,a_2,\dots,a_n)
\]
(see~\cite{Dolgachev82} as a general reference on the subject).

\subsection{Curves in quotient surface singularities}
\label{sec:quotient}

We introduce some notation for germs of curves in a quotient surface singularity
$S:=\frac{1}{d}(a,b)$ (with $a,b,d$ pairwise coprime and $d>1$). Let $\pi:\CC^2\to S$ be the quotient
map. Any germ of curve $C\subset S$ is defined as the zero locus of a non-constant equivariant germ
$f\in\CC\{x,y\}$, that is, a germ satisfying $f(\zeta\cdot (x,y))=\zeta^k f(x,y)$ for some $k=0,...,d-1$.
For a fixed $k$, the collection of all such equivariant germs inherits an $\mathcal O_S$-module structure
as a subset of $\CC\{x,y\}$ and will be denoted by $\mathcal O_S(k)$.
Note that an equivariant germ is a function on $S$ only when~$k=0$, that is,~$\mathcal O_S=\mathcal O_S(0)$.

\begin{defn}
A germ of curve~$C$ is said to be \emph{quasi-smooth} if $C$ is smooth as an abstract curve. 
If, in addition, a defining germ for $C$ can be found to have multiplicity one, then $C$ is said to be \emph{extremely quasi-smooth}. 
\end{defn}

\begin{rem}
There are simple characterizations of the above concepts in terms of a minimal resolution
$\hat{S}\to S$; recall that its dual graph is a bamboo whose vertices represent smooth rational divisors.
A curve is quasi-smooth if its strict transform in~$\hat{S}$ is a curvette of an exceptional divisor, that is,
smooth and transversal to it at a smooth point of the exceptional locus. 
Moreover, it is extremely quasi-smooth if this divisor is either end of the bamboo. 
In the particular case $\frac{1}{d}(1,1)$, any quasi-smooth curve is extremely quasi-smooth, and any linear form 
can be the multiplicity-one component of~$f$. Otherwise, in $\frac{1}{d}(a,b)$ with $(a,b)\neq (1,1)$ the equivariant 
part of multiplicity 1 of an extremely quasi-smooth $f$ can only be given by the eigenspaces of the cyclic action, 
in our notation, either $x$ or~$y$.
\end{rem}

\subsection{Weighted projective planes}
\label{sec:WPR}
In this section we briefly describe weighted projective planes in order to fix some notation.
A \emph{weight} is a triple $\omega:=(e_1,e_2,e_3)\in\ZZ_{>0}^3$ such that $\gcd\omega=1$.
The \emph{weighted projective plane} $\PP^2_\omega$ is a normal surface obtained as the quotient
of $\CC^3\setminus\{0\}$ by the action of $\CC^*$ given by
\[
t\cdot(x,y,z)=(t^{e_1} x, t^{e_2} y, t^{e_3} z),\qquad t\in\CC^*, (x,y,z)\in\CC^3\setminus\{0\}.
\]
Weighted projective lines are defined in a similar way.
The symbol $[x:y:z]_\omega$ stands for points in $\PP^2_\omega$, for orbits in 
$\CC^3\setminus\{0\}$ or their closure in $\CC^3$. This variety is covered
by three \emph{quotient charts}. One of them is
\[
\begin{tikzcd}[row sep=0]
\frac{1}{e_3}(e_1,e_2)\ar[r,"\Psi_{\omega,3}"]&\PP^2_\omega\setminus\{z=0\}\\
{[(x,y)]}\ar[r,mapsto]&{[x:y:1]}_\omega .
\end{tikzcd}
\]
The other two quotient charts are defined accordingly.

Define $d_k:=\gcd(e_i,e_j)$ and $\alpha_k:=\frac{e_k}{d_id_j}$, $\{i,j,k\}=\{1,2,3\}$. 
Note that $\eta:=(\alpha_1,\alpha_2,\alpha_3)$ are pairwise coprime. According to the 
properties described above, the map 
\begin{equation}
\label{eq:P2isom}
\begin{tikzcd}[row sep = 0ex
     ,/tikz/column 1/.append style={anchor=base east}
     ,/tikz/column 2/.append style={anchor=base west}]
\PP^2_\omega\arrow[r,"\pi_{\eta,\omega}"]&\PP^2_\eta\\
{[x:y:z]}_\omega\arrow[r,mapsto]&{[x^{d_1}:y^{d_2}:z^{d_3}]}_\eta
\end{tikzcd}
\end{equation}
is well defined since 
\[
t\cdot\!{[x\!:\!y\!:\!z]}_\omega\!=\![t^{e_1}\!x\!:\!t^{e_2}\!y\!:\!t^{e_3}\!z]_\omega\!\mapsto\!
[t^{d_1 e_1}\!x^{d_1}\!:\!t^{d_2 e_2}\!y^{d_2}\!:\!t^{d_3 e_3}\!z^{d_3}]_\eta\!=\!t^{d_1 d_2 d_3}\cdot\!{[x^{d_1}\!:\!y^{d_2}\!:\!z^{d_3}]}_\eta,
\]
and $d_i e_i=\alpha_i d_1 d_2 d_3$.
Moreover, one can easily check that it is an isomorphism.

One may consider $\PP^2_\omega$ and $\PP^2_\eta$ in a slightly different way (see also~\cite{Dolgachev82}).
The plane $\PP^2_\eta$ has at most 3 singular points at $P_x:=[1:0:0]_\eta$ (if $\alpha_1>1$), $P_y:=[0:1:0]_\eta$ 
(if $\alpha_2>1$), and $P_z:=[0:0:1]_\eta$ (if $\alpha_3>1$). The plane $\PP^2_\omega$ is 
an orbifold where the quotient charts $\Psi_{\omega,i}$ are not normalized; the associated analytic variety to $\PP^2_\omega$ is 
$\PP^2_\eta$ since the normalization of the source of $\Psi_{\omega,i}$ is precisely the source of~$\Psi_{\eta,i}$.

\subsection{Weighted blow-ups}\label{sec:wblowups}

Let us consider now $\omega:=(e_1,e_2)\in\ZZ^2_{>0}$, $\gcd\omega=1$.
The $\omega$-\emph{weighted blow-up} of~$\CC^2$ at the origin is the map
$\pi_\omega:\widehat{\CC}^2_\omega\to\CC^2$ where
\[
\widehat{\CC}^2_\omega:=\{(\mathbf{x},\mathbf{u})\in\CC^2\times\PP^1_\omega\mid \mathbf{x}\in\mathbf{u}\}.
\]
This normal variety is represented with two quotient charts. One of them is
\[
\widehat{\Psi}_{\omega,2}:\frac{1}{e_2}(e_1,-1)\to\widehat{\CC}^2_\omega,\qquad
(x,y)\mapsto((xy^{e_1},y^{e_2}),[x:1]_\omega);
\]
the other one is analogous and modeled on $\frac{1}{e_1}(-1,e_2)$. The exceptional 
divisor of $\pi_\omega$ is a weighted projective line which contains the singular
points $(0,[1:0]_\omega)$ (if $e_1>1$) and $(0,[0:1]_\omega)$ (if $e_2>1$)
of the surface~$\widehat{\CC}^2_\omega$. Note that the curvettes of this divisor
are extremely quasi-smooth if either $e_1$ or $e_2$ equal~$1$.

Let us study now $3$-dimensional weighted blow-ups. We recover the notation introduced in
\S\ref{sec:WPR}
for a weight $\omega$ and its normalization $\eta$, both in $\mathbb{Z}^3_{>0}$. 
We consider $\Pi_\omega:\widehat{\CC}^3_\omega\to\CC^3$ where
\[
\widehat{\CC}^3_\omega:=\{(\mathbf{x},\mathbf{u})\in\CC^3\times\PP^2_\omega\mid \mathbf{x}\in\mathbf{u}\}.
\]
The normal variety is now represented with three charts. One of them is
\[
\widehat{\Psi}_{\omega,3}:\frac{1}{e_3}(e_1,e_2,-1)\to\widehat{\CC}^3_\omega,\qquad
(x,y,z)\mapsto((xz^{e_1},yz^{e_2},z^{e_3}),[x:y:1]_\omega);
\]
the other two charts can analogously be defined and have as domains the quotients 
$\frac{1}{e_1}(-1,e_2,e_3)$ and $\frac{1}{e_2}(e_1,-1,e_3)$.

Let us study the local structure of $\widehat{\CC}^3_\omega$ at $E_\omega:=\Pi_\omega^{-1}(0)$; since
$\Pi_\omega$ is an isomorphism outside this exceptional divisor the points not in $E_\omega$ 
are smooth. Note that $E_\omega$ is naturally isomorphic to $\PP^2_\omega$; in addition
by \eqref{eq:P2isom}, one has $\PP^2_\omega\cong\PP^2_\eta$. 
For the sake of simplicity
we will denote the elements of $E_\omega$ only by their $\omega$-quasi-homogeneous coordinates.
Let us denote: 
\begin{gather*}
P_x=[1:0:0]_\omega,\quad P_y=[0:1:0]_\omega,\quad P_z=[0:0:1]_\omega,\\
\check{X}=\{[0:y:z]_\omega\mid y z\neq 0\},\
\check{Y}=\{[x:0:z]_\omega\mid x z\neq 0\},\
\check{Z}=\{[x:y:0]_\omega\mid x y\neq 0\}.
\end{gather*}

In order to provide a  stratification of $E_\omega$ in terms of the singular points of the ambient space we need
a description of the singular locus.

\begin{prop}\label{prop:singular_locus}
Let $P=[x_0:y_0:1]_\omega\in E_\omega\cap\widehat{\Psi}_{\omega,3}(\CC^3)\subset\widehat{\CC}^3_\omega$. The following properties hold:
\begin{enumerate}[label=\rm(\arabic{enumi})]
 \item 
 If $x_0y_0\neq0$ then $(\widehat{\CC}^3_\omega,P)$ is smooth.
 \item\label{prop:singular_locus2} 
 If $P\in\check{X}$, i.e. $y_0\neq0$ and $x_0= 0$, then $(\widehat{\CC}^3_\omega,P)$ is isomorphic
 to the germ at the origin of $\frac{1}{d_1}(e_1,0,-1)$.
 \item\label{prop:singular_locus3} 
 If $P\in\check{Y}$, i.e. $x_0\neq0$ and $y_0= 0$, then $(\widehat{\CC}^3_\omega,P)$ is isomorphic
 to the germ at the origin of $\frac{1}{d_2}(0,e_2,-1)$.
 \item If $P=P_z$, i.e. $x_0=y_0= 0$, then $(\widehat{\CC}^3_\omega,P_z)$ is isomorphic
 to the germ at the origin of~$\frac{1}{e_3}(e_1,e_2,-1)$.
\end{enumerate}
\end{prop}

\begin{proof}
It is only necessary to prove~\ref{prop:singular_locus2}. 
Note that $P$ is obtained as the image by $\widehat{\Psi}_{\omega,3}$ of $(0,y_0,0)\in\frac{1}{e_3}(e_1,e_2,-1)$. The isotropy
subgroup of $(0,y_0,0)$ by the action is the cyclic group of order $d_1=\gcd(e_2,e_3)$. 
Hence at a neighborhood of $(0,y_0,0)$ the space looks like $\frac{1}{d_1}(e_1,e_2,-1)=\frac{1}{d_1}(e_1,0,-1)$.
\end{proof}

\begin{rem}
A similar statement holds for the other charts. Note that a point satisfying property~\ref{prop:singular_locus2}
above, say $P=[0:1:1]_\omega$ 
belongs in the image of $\widehat{\Psi}_{\omega,3}$, $P=\widehat{\Psi}_{\omega,3}(0,1,0)$ as stated in 
Proposition~\ref{prop:singular_locus}, but also in the image of $\widehat{\Psi}_{\omega,2}$, 
$P=\widehat{\Psi}_{\omega,2}(0,0,1)$. Note that the notation for the quotient types given above do 
no match, that is, $\frac{1}{d_1}(e_1,0,-1)$ if considered in $\widehat{\Psi}_{\omega,3}(\CC^3)$ and 
$\frac{1}{d_1}(e_1,-1,0)$ if considered in $\widehat{\Psi}_{\omega,3}(\CC^3)$.
To avoid this ambiguity we will simply say that given $P\in\check{X}$, then $(\widehat{\CC}^3_\omega,P)$ is isomorphic
to the product of $(\CC,0)$ and the germ at the origin of $\frac{1}{d_1}(e_1,-1)$.
A similar property holds for $\check{Y},\check{Z}$.
\end{rem}

\begin{rem}
If $d_1=e_3$, i.e., if $e_3$ divides $e_2$, in the proof of Proposition~\ref{prop:singular_locus} 
the condition $y_0\neq 0$ is not needed and $P_z$ behaves as the points in $\check{X}$. 
A similar property holds for the other pairs of axes and vertices.
\end{rem}

\begin{notation}\label{ntc:strata}
We fix the following notation for the strata of $E_\omega$.
\begin{itemize}
\item 
2-dimensional stratum.
The stratum $\mathcal{T}$ is the intersection of $E_\omega$ with the smooth subvariety of $\widehat{\CC}^3_\omega$;
it contains $\{[x:y:z]_\omega\mid xyz\neq 0\}$. It contains also
$\check{X}$ (resp.~$\check{Y}$, resp.~$\check{Z}$) if $d_1=1$ (resp.~$d_2=1$, resp.~$d_3=1$) and 
$P_x$ (resp.~$P_y$, resp.~$P_z$) if $e_1=1$ (resp.~$e_2=1$, resp.~$e_3=1$).

\item 
1-dimensional strata.  Following Proposition~\ref{prop:singular_locus} and the above remarks we set:
$$\mathcal{L}_x=
\begin{cases} 
\emptyset & \text{ if } d_1=1\\ 
\check{X}\cup \{P_y\} & \text{ if } 1<d_1=e_2\neq e_3\\
\check{X}\cup \{P_z\} & \text{ if } 1<d_1=e_3\neq e_2\\
X & \text{ if } 1<d_1=e_2=e_3\\
\check{X} & \text{ otherwise. }
\end{cases}
$$
The remaining strata $\mathcal{L}_y$ and $\mathcal{L}_z$ are defined accordingly.

\item 
0-dimensional strata. 
$$
\mathcal{P}_x=\begin{cases} \emptyset & \text{ if } e_1\text{ divides either } e_2\text{ or }e_3,\\ \{P_x\} & \text{ otherwise.}\end{cases}
$$
The remaining strata $\mathcal{P}_y$
and $\mathcal{P}_z$ are defined accordingly.

\end{itemize}
\end{notation}

\subsection{Weighted Cremona transformations}\label{sec:cremona}

The most well-known Cremona transformation of $\PP^2$ corresponds to
the birational map $[x:y:z]\mapsto[yz:xz:xy]$; geometrically, this map
is the composition of the blow-ups at $[1:0:0],[0:1:0],[0:0:1]$
and the contractions of the strict transforms of the lines 
$x=0,y=0,z=0$ which become pairwise disjoint $(-1)$-lines in the blown-up
plane.

In this section we generalize this transformation to a birational map from a 
weighted projective plane to~$\PP^2$.
Let us fix $\PP^2_\omega$, $\omega:=(e_1,e_2,e_3)$, where 
$e_1,e_2,e_3$ are pairwise coprime, i.e., $\omega=\eta$. In order to stress this property we will use the notation
$e_i=\alpha_i$, $i=1,2,3$. 
Consider two positive integers $\beta_1,\beta_2$ such that $\alpha_1\beta_1+\alpha_2\beta_2=\alpha_3 + \alpha_1 \alpha_2$ (they exist from
standard semigroup properties). These arithmetic data provide the following map 
\[
\begin{tikzcd}[ /tikz/column 1/.append style={anchor=base east},/tikz/column 3/.append style={anchor=base west}, row sep=0pt]
\mathbb{P}^2_\omega\arrow[rr,dashed,"\Phi_{\omega,\beta_1,\beta_2}"]&&\mathbb{P}^2\\
{[x:y:z]}_\omega\arrow[rr,mapsto]&&{[y^{\alpha_1}  z: x^{\alpha_2}  z: x^{\beta_1} y^{\beta_2}]},
\end{tikzcd}
\]
which is a well-defined rational map (not a morphism) since the three coordinates have $\omega$-degree equal to $\alpha_1 \alpha_2+\alpha_3 $.
It is in fact a birational map whose inverse is given by
\[
\begin{tikzcd}[ /tikz/column 1/.append style={anchor=base east},/tikz/column 2/.append style={anchor=base west}, row sep=0pt]
\mathbb{P}^2\arrow[r,dashed]&\mathbb{P}^2_\omega\\
{[x:y:z]}\arrow[r,mapsto]&{\left[y^{\frac{1}{\alpha_2}} z^{\frac{\alpha_1}{\alpha_3}}: x^{\frac{1}{\alpha_1}} z^{\frac{\alpha_2}{\alpha_3}}: x^{\frac{\beta_2}{\alpha_1}} y^{\frac{\beta_1}{\alpha_2}}\right]}_\omega.
\end{tikzcd} 
\]
We will show that this map is well defined as long as the radicals
$x^{\frac{1}{\alpha_1}},y^{\frac{1}{\alpha_2}},z^{\frac{1}{\alpha_3}}$ are chosen consistently throughout the formula. 
Assume $x_0$ (resp.~$y_0$, $z_0$) is such that $x_0^{\alpha_1}=x$ (resp.~$y_0^{\alpha_2}=y$, $z_0^{\alpha_3}=z$) 
and choose for instance $x_1=\zeta_{\alpha_1} x_0$. Let $\hat{\alpha}_2\in\mathbb{Z}$ be such that $\alpha_2\hat{\alpha}_2\equiv 1\bmod{\alpha_1}$. As a consequence, the following congruences hold: 
$\alpha_3\hat{\alpha}_2\equiv(\alpha_3 + \alpha_1 \alpha_2)\hat{\alpha}_2\equiv(\alpha_1\beta_1+\alpha_2\beta_2)\hat{\alpha}_2\equiv\beta_2\bmod{\alpha_1}$. Then 
\begin{gather*}
[y_0z_0^{\alpha_1}\!:\!x_1z_0^{\alpha_2}\!:\!x_1^{\beta_2}y_0^{\beta_1}]_\omega\!=\!
[y_0z_0^{\alpha_1}:\zeta_{\alpha_1}x_0z_0^{\alpha_2}:\zeta_{\alpha_1}^{\beta_2}x_0^{\beta_2}y_0^{\beta_1}]_\omega\!=\\
[(\zeta_{\alpha_1}^{\hat{\alpha}_2})^{\alpha_1}y_0z_0^{\alpha_1}:(\zeta_{\alpha_1}^{\hat{\alpha}_2})^{\alpha_2}x_0z_0^{\alpha_2}:(\zeta_{\alpha_1}^{\hat{\alpha}_2})^{\alpha_3}x_0^{\beta_2}y_0^{\beta_1}]_\omega\!=
[y_0z_0^{\alpha_1}:x_0z_0^{\alpha_2}:x_0^{\beta_2}y_0^{\beta_1}]_\omega.
\end{gather*}
A similar argument applies to other choices of roots of $y^{\frac{1}{\alpha_2}}$ and $z^{\frac{1}{\alpha_3}}$. 
These equations completely determine the birational map, but a more geometric description will be useful.

\begin{prop}
The map  $\Phi_{\omega,\beta_1,\beta_2}$ is the composition of the following blow-ups and downs:
\begin{enumerate}[label=\rm(\arabic{enumi})]
 \item Three simultaneous blow-ups:
 \begin{enumerate}[label=\rm(\alph{enumii})]
  \item Type $(\alpha_1,\alpha_2)$ at $[0:0:1]_\omega\cong\frac{1}{\alpha_3}(\alpha_1,\alpha_2)$.
  \item Type $(1,\beta_1)$ at $[0:1:0]_\omega$ isomorphic to 
  \[
\frac{1}{\alpha_2}(\alpha_1,\alpha_3)\!=\!\frac{1}{\alpha_2}(\alpha_1,\alpha_1 \alpha_2+\alpha_3 )\!=\!\frac{1}{\alpha_2}(\alpha_1,\alpha_1\beta_1+\alpha_2\beta_2)\!=\!\frac{1}{\alpha_2}(1,\beta_1).
  \]
\item Type $(1,\beta_2)$ at $[1:0:0]_\omega$ isomorphic to 
  \[\frac{1}{\alpha_1}(\alpha_2,\alpha_3)\!=\!\frac{1}{\alpha_1}(\alpha_2,\alpha_1 \alpha_2+\alpha_3 )\!=\!\frac{1}{\alpha_1}(\alpha_2,\alpha_1\beta_1+\alpha_2 \beta_2)\!=\!\frac{1}{\alpha_1}(1,\beta_2).
  \]
 \end{enumerate}
 \item Three simultaneous blow-downs:
\begin{enumerate}[label=\rm(\alph{enumii})]
\item Type $(1,1)$ at $[0:0:1]$.
\item Type $(\alpha_2,\beta_1)$ at $[1:0:0]$.
\item Type $(\alpha_1,\beta_2)$ at $[0:1:0]$.
  \end{enumerate}
\end{enumerate}

\end{prop}

\begin{proof}
Let us start with the three blow-ups in $\PP^2_\omega$. We obtain a normal rational surface~$S$. The preimage of
the three axes appear in 
Figure~\ref{fig:cremona}, containing the strict transforms $L_x,L_y,L_z$ of the lines and the exceptional components $E_x,E_y,E_z$.
The self-intersections and the type of the singular points are  computed using~\cite[Theorem 4.3]{AMO-Intersection}.

\begin{figure}[ht]
\begin{center}
\begin{tikzpicture}[scale=.89,vertice/.style={draw,circle,fill,minimum size=0.2cm,inner sep=0}]
\tikzset{%
  suma/.style args={#1 and #2}{to path={%
 ($(\tikztostart)!-#1!(\tikztotarget)$)--($(\tikztotarget)!-#2!(\tikztostart)$)%
  \tikztonodes}}
}

\begin{scope}[xshift=-4cm,yscale=-1,yshift=-1cm]
\coordinate (X) at (0,2);
\coordinate (Z) at (-1,-1);
\coordinate (Y) at (2.5,-1);
\coordinate (Y1) at (3.55,.4);
\coordinate (X1) at (-2.35,.5);
\coordinate (XZ) at (-1,2);
\coordinate (XY) at (2.5,2);
\draw[suma=.2 and .2,line width=.5] (Z)node[above left] {$\frac{1}{\alpha_1}(\beta_2,-1)$} to node[pos=.5,above] {$-\frac{\alpha_3}{\alpha_1 \alpha_2}$} node[below] {$E_z$}(Y) node[above] {\ $\frac{1}{\alpha_2}(\beta_1,-1)$} ;
\draw[suma=.35 and .35,line width=.5] (XZ) to node[pos=.6,right] {$-\frac{\alpha_1}{\beta_2}$} node[below left=-1pt] {$E_x$} (X1) ;
\draw[suma=.35 and .35,line width=1.2] (Z) to node[right,pos=.5] {$-\frac{1}{\alpha_1\beta_2}$} node[pos=.82,left] {$\frac{1}{\beta_2}(\alpha_1,-1)$} node[pos=.5,left] {$L_y$}(X1) ;
\draw[suma=.35 and .35,line width=1.2] (Y) to node[left,pos=.5] {$-\frac{1}{\alpha_2\beta_1}$} node[pos=.81,right=2pt] {$\frac{1}{\beta_1}(\alpha_2,-1)$} node[pos=.5,right] {$L_x$}(Y1);
\draw[suma=.35 and .35,line width=.5] (XY) to node[left,pos=.45] {$-\frac{\alpha_2}{\beta_1}$} node[right,pos=.45] {$E_y$} (Y1);
\draw[suma=.2 and .2,line width=1.2] (XZ) to node[above, pos=.5]  {$-1$} node[below] {$L_z$}(XY);
\node[vertice] at (Y) {} ;
\node[vertice] at (Z) {} ;
\node[vertice] at (X1) {};
\node[vertice] at (Y1) {};
\node[vertice] at (XZ) {};
\node[vertice] at (XY) {};

\end{scope}

\end{tikzpicture}
\caption{Weighted blow-ups of $\PP^2$ in $S$}
\label{fig:cremona}
\end{center}
\end{figure}
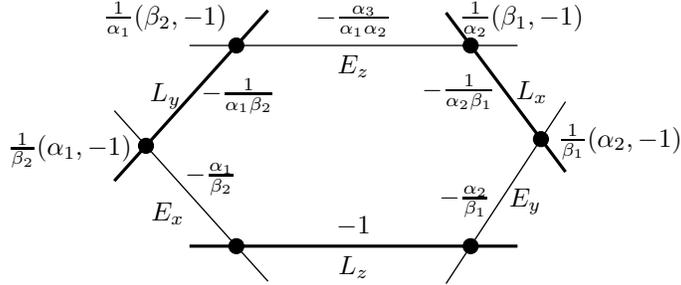
The strict transforms of the lines coincide with the exceptional components of a 
$(\alpha_1,\beta_2)$-blowing-up ($L_y$), a $(\alpha_2,\beta_1)$-blowing-up ($L_x$) and a 
standard blowing-up ($L_z$). The result of the triple blowing-down is $\PP^2$.
\end{proof}

This geometric expression will be useful for the study of curves in $\PP^2_\omega$ via their transforms in $\PP^2$.

\section{Zariski pairs on weighted projective planes}\label{sec:zp}

In this section, we are going to use the Cremona transformations in \S\ref{sec:cremona} to produce
Zariski pairs in weighted projective planes. By a \emph{Zariski pair} we mean two curves embedded in the 
same surface whose \emph{combinatorics} are the same, but whose embeddings are non-homeomorphic. 
As in the classical case of curves in the projective plane, the combinatorics of a curve in a weighted 
projective plane is encoded by the degrees of its irreducible components and the dual graph of a minimal 
resolution of the curve (where the strict transforms of the irreducible components of the curve are marked).

In this section we will produce families of Zariski pairs of irreducible curves.
Let us start with the combinatorics defined by a smooth projective cubic and three tangent lines at inflection points.
Note that a generic choice of a smooth cubic can be made so that such lines are non-concurrent and hence the 
remaining singular points are three nodes. 
This combinatorics admits a Zariski pair of sextics, see~\cite{Artal94}, and their embeddings are distinguished by the 
algebraic property of whether or not the inflection points of the cubic, that is, the three non-nodal 
singular points of the sextic, which have type $\mathbb{A}_6$, are aligned.
The image by a standard Cremona transformation of the smooth cubics (using the three tangent lines at the axes)
produces a Zariski pair of irreducible sextics with three $\mathbb{E}_6$-points. In this case, the embeddings 
can be proven to be different showing that the fundamental group of their complements are not isomorphic.

Our strategy is to replace this Cremona transformation by the inverse of those described in~\S\ref{sec:cremona}.

\subsection{Fundamental groups of complements}
\label{sec:fund-groups}
Let us start by recalling the two possible fundamental groups of the complements of the sextic curves given as the 
union of a smooth cubic and three tangent lines at inflection points.

\begin{prop}[\cite{ac:98}]
\label{prop:pi1cubic1}
Let  $\mathcal{C}$ be a smooth cubic with three tangent lines $X,Y,Z$ at inflections which are \emph{not} aligned. 
Then, $\pi_1(\PP^2\setminus(\mathcal{C}\cup X\cup Y \cup Z))$ is abelian.
\end{prop}

In \cite{ac:98}, the fundamental group of the other member of the Zariski pair is also computed; 
since it is non-abelian, this invariant distinguishes the two members. For our purpose, we need 
a more geometrical presentation of the group involving meridians for all the irreducible components
and such that the meridians close to the nodes are made explicit.
Let us recall the concept of meridian in order to clarify what we mean by 
\emph{meridians close to a singular point}.

\begin{defn}\label{def-mrd} Let $Z$ be a connected quasi-projective manifold
and let $H$ be a hypersurface of $Z$. Consider $P\in Z\setminus H$ and $K$ an irreducible component of $H$. 
A homotopy class $\gamma\in\pi_1(Z\setminus H;P)$ is called a \emph{meridian about $K$ with respect to $H$}
if $\gamma=[\delta]$ for some loop $\delta$ satisfying the following:
\begin{enumerate}[label=\rm(\arabic*)]
\item there is a smooth complex analytic disk
$\Delta \subset Z$ transverse to $H$ such that
$\Delta\cap H=\{P'\} \subset K$
(transversality implies that $P'$ is a smooth point of $H$).

\item 
there is a path $\alpha$ in $Z\setminus H$ starting at $P$ and ending at some point $P'' \in \partial\Delta$.

\item $\delta=\alpha*\beta*\overline{\alpha}$, where the operation $*$ here means concatenation of paths 
from left to right, $\beta$ is the closed path obtained by traveling from $P''$ along $\partial\Delta$ in 
the positive direction and $\overline{\alpha}$ represents the path $\alpha$ traveled in the opposite direction, 
that is, $\overline{\alpha}(t):=\alpha(1-t)$.
\end{enumerate}
\end{defn}

It is well known that meridians with respect to the same irreducible component 
define a conjugacy class of members of the fundamental group.

\begin{ex}
Let $Z=\CC^2$ and $H=\{x y=0\}$ and let $P:=(1,1)$. The paths
$\mu_x,\mu_y:[0,1]\to Z\setminus H$ defined by 
\[
\mu_x(t)=(e^{2 i\pi t},1),\quad 
\mu_y(t)=(1,e^{2 i\pi t}),
\]
define meridians with respect to the irreducible components of $H$ (for which the path $\alpha$ is trivial).
They commute as elements in the fundamental group $\pi_1(Z\setminus H;P)$. If $Z$ is quasi-projective surface and $H$ 
is a curve containing a node, \emph{two meridians are close to the node} if there is a common path $\alpha$ from the base 
point of $\pi_1(Z\setminus H;P)$ to a point \emph{close} to the node such that the $\beta$-paths look like in this example.
\end{ex}

\begin{prop}[\cite{ACO-Kummer}]
Let  $\mathcal{C}$ be a smooth cubic with three tangent lines $X,Y,Z$ at inflections which are  aligned. Then,
$\pi_1(\PP^2\setminus(\mathcal{C}\cup X\cup Y \cup Z))$ is
\begin{equation}\label{eq:grupoP2}
\langle
c,\ell_x,\ell_y,\ell_z\ |\ 
[\ell_x,\ell_y]\!=\![\ell_y,\ell_z]\!=\![\ell_z,\ell_x]\!=\![c,\ell_x^{-1}\ell_z]
\!=\![c,\ell_y^{-1}\ell_z]\!=\!c\ell_x c\ell_y c\ell_z\!=\!\!1
\rangle
\end{equation}
where $c$ is a meridian of $\mathcal{C}$, and $\ell_x,\ell_y,\ell_z$ are meridians of $X,Y,Z$, respectively; 
moreover the meridians of the lines correspond to meridians close to the double points.
\end{prop}

Let us fix $\Phi:=\Phi_{\omega,\beta_1,\beta_2}$ as in~\S\ref{sec:cremona}, 
and let us denote by $\tilde{\mathcal{C}}\subset \mathbb{P}^2_\omega$
the strict transform of the smooth cubic $\mathcal{C}$ by $\Phi$, where 
the lines $X,Y,Z$ have equations $x=0,y=0,z=0$, respectively. Consider the following three homogeneous 
polynomials of degree~3
\[
H_\lambda(x,y,z):=x^3+y^3+z^3+3 x y (\lambda^{-1} x+\lambda y)+3 x z (x+z)+3 y z (\lambda^{-1} y+\lambda z),
\]
where $\lambda^3=1$. The curve $\mathcal{C}_\lambda=\{H_\lambda=0\}$ is a smooth cubic which is tangent to the 
line $L_x$ at the inflection point $[0:1:-\lambda]$ and analogously for $Y$ at $[-1:0:1]$, and $Z$ at $[1:-\lambda:0]$. 
Note that for the cubic $\mathcal{C}_1$ the three inflection points are contained in the line $x+y+z=0$. 
However, for the smooth cubic $\mathcal{C}_{\exp {\frac{2i\pi}{3}}}$ the three inflection points are not aligned.

\begin{cor}\label{cor:abelian}
In the non-aligned case, $\pi_1(\PP^2_\omega\setminus\tilde{\mathcal{C}})$ is isomorphic to $\ZZ/3(\alpha_1 \alpha_2+\alpha_3 )$.
\end{cor}

\begin{proof}
The space $\PP^2_\omega\setminus\tilde{\mathcal{C}}$ is homeomorphic to 
$S\setminus(\widehat{\mathcal{C}}\cup L_x\cup L_y\cup L_z)$ (see Figure~\ref{fig:cremona}) and
the space $\PP^2\setminus(\mathcal{C}\cup X\cup Y \cup Z)$
is homeomorphic to $S\setminus(\widehat{\mathcal{C}}\cup L_x\cup L_y \cup L_z\cup E_x\cup E_y\cup E_z)$,
where $\widehat{\mathcal{C}}$ denotes the strict transform of $\mathcal{C}$ in~$S$.
As a consequence of \cite[Lemma 4.18]{fujita:82} the kernel of the epimorphism
\begin{equation}
\label{eq:pi1surj}
\pi_1(\PP^2\setminus(\mathcal{C}\cup X\cup Y \cup Z))\twoheadrightarrow
\pi_1(\PP^2_\omega\setminus\tilde{\mathcal{C}})
\end{equation}
is the normal subgroup generated by the meridians of $E_x,E_y,E_z$ in $S$. Since the source is 
an abelian group by Proposition~\ref{prop:pi1cubic1}, the group $\pi_1(\PP^2_\omega\setminus\tilde{\mathcal{C}})$ 
is abelian as well. Hence it coincides with 
$H_1(\PP^2_\omega\setminus\tilde{\mathcal{C}};\ZZ)\cong\ZZ/\deg(\tilde{\mathcal{C}})$,
since $\tilde{\mathcal{C}}$ contains the vertices of~$\PP^2_\omega$.
\end{proof}

In order to compute the other fundamental group we need a technical result.

\begin{lem}\label{lem:blow}
Let $\pi:\widehat{\mathbb{C}}^2_{(\alpha_1,\alpha_2)}\to\mathbb{C}^2$ be the $(\alpha_1,\alpha_2)$-blow-up of the origin in $\mathbb{C}^2$ and let 
$E$ denote its exceptional component. Let $X,Y\subset\mathbb{C}^2$ be the axes (curves of equations $x=0$, $y=0$, respectively), 
and let us keep this notation for their strict transforms. Let
$U:=\mathbb{C}^2\setminus(X\cup Y)\equiv\widehat{\mathbb{C}}^2_{(\alpha_1,\alpha_2)}\setminus(E\cup X\cup Y)$.

If $\mu_X,\mu_Y,\mu_E$ denote meridians of the respective curves in $\pi_1(U)\cong\ZZ\mu_X\oplus\ZZ\mu_Y$, then
(multiplicative notation) $\mu_E=\mu_X^{\alpha_1} \mu_Y^{\alpha_2}$.
\end{lem}

\begin{proof}
Consider $(1,1)$ as the base point, then $\mu_X$ is the loop $t\mapsto(e^{2i\pi t},1)$, while $\mu_Y$ is the loop 
$t\mapsto(1,e^{2i\pi t})$. Let us pick a chart of $\widehat{\mathbb{C}}^2_{(\alpha_1,\alpha_2)}$, say
\[
\begin{tikzcd}[row sep=0pt]
\frac{1}{\alpha_1}(-1,\alpha_2)\arrow[r]&\mathbb{C}^2\\
{[(x,y)]}\arrow[r,mapsto]&(x^{\alpha_1}, x^{\alpha_2}  y).
\end{tikzcd}
\]
The base point in the chart is the class of $(1,1)$; the equation of $E$ is $x=0$ and hence $\mu_E$ is represented by 
$t\mapsto[(e^{2i\pi t},1)]$. Hence, in $\mathbb{C}^2$ is represented by $t\mapsto(e^{2i \alpha_1\pi t},e^{2i \alpha_2\pi t})$
and the result follows.
\end{proof}

\begin{prop}\label{prop:group_cubic}
In the aligned case, $\pi_1(\PP^2_\omega\setminus\tilde{\mathcal{C}})$ is isomorphic to 
$\ZZ/3(\alpha_1 \alpha_2+\alpha_3 )$ if $2$ divides $\alpha_1 \alpha_2 \alpha_3 \beta_1\beta_2$ and to
\begin{equation}
\label{eq:pres-odd}
\langle
\ell,u
\ |\ 
\ell^{\alpha_1 \alpha_2+\alpha_3 }=1,u^3=\ell^{2}
\rangle
\end{equation}
otherwise. This group is a central extension of
$\mathbb{Z}/2*\mathbb{Z}/3$ by a cyclic group of order~$\frac{\alpha_1 \alpha_2+\alpha_3 }{2}$.
\end{prop}

\begin{proof}
Following the proof of Corollary~\ref{cor:abelian}, the epimorphism described in~\eqref{eq:pi1surj}
also holds in this case. Hence a presentation of $\pi_1(\PP^2_\omega\setminus\tilde{\mathcal{C}})$ can be given once
meridians of $E_x,E_y,E_z$ are written in terms of the generators provided in~\eqref{eq:grupoP2}. 
Since the meridians $\ell_x,\ell_y,\ell_z$ of the lines in the presentation~\eqref{eq:grupoP2} are homotopic to 
meridians close to the double points, by Lemma~\ref{lem:blow} we have that $\ell_x \ell_y$ is a meridian of $E_z$, 
$\ell_x^{\alpha_1}\ell_z^{\beta_2}$ is a meridian of $E_x$, and $\ell_y^{\alpha_2}\ell_z^{\beta_1}$ is a meridian of $E_y$.
Hence a presentation of $\pi_1(\PP^2_\omega\setminus\tilde{\mathcal{C}})$ can be obtained by adding the relations
\begin{equation}
\label{eq:rels-meridians} 
\ell_x \ell_y=\ell_x^{\alpha_1}\ell_z^{\beta_2}=\ell_y^{\alpha_2}\ell_z^{\beta_1}=1
\end{equation}
to the presentation given in~\eqref{eq:grupoP2}.

Finally, let us simplify this presentation. As a first step one can eliminate $\ell_x$, since $\ell_x=\ell_y^{-1}$. 
Also, choose $\hat{\alpha}_1,\hat{\alpha}_2\in\ZZ$ such that $\alpha_2\hat{\alpha}_1-\alpha_1\hat{\alpha}_2=1$. 
Note that $\ell_y,\ell_z$ commute; then the 
remaining two relations in~\eqref{eq:rels-meridians} become
\[
\ell_y^{-\alpha_1}\ell_z^{\beta_2}=\ell_y^{\alpha_2}\ell_z^{\beta_1}=1\Longrightarrow
\begin{cases}
1=\ell_z^{\alpha_1\beta_1+\alpha_2\beta_2}= \ell_z^{\alpha_1 \alpha_2+\alpha_3 }, \\
\ell_y=\ell_z^{-(\hat{\alpha}_1\beta_2+\hat{\alpha}_2\beta_1)}.
\end{cases}
\]
In fact, this is an equivalence. Let us denote $\ell:=\ell_z$ and $u:=c\ell$. 
Since $[c,\ell_y\ell]=[c,\ell_y^{-1}\ell]=1$, one has
\[
\begin{aligned}
1=&c\ell_y^{-1} c\ell_y c\ell=c\ell_y^{-1}c \ell_y\ell \ell^{-1}c\ell=c\ell_y^{-1}(\ell_y\ell)c\ell^{-1}c\ell 
\Longleftrightarrow \\
1=&(c\ell)^2c\ell^{-1}=(c\ell)^3\ell^{-2}\Longleftrightarrow 
\ell^2=u^3.
\end{aligned}
\]
Hence $\pi_1(\PP^2_\omega\setminus\tilde{\mathcal{C}})$ admits a presentation
\begin{equation}
\label{eq:pres1}
\langle
\ell,u
\ |\ 
\ell^{\alpha_1 \alpha_2+\alpha_3 }=1,[u,\ell^{\hat{\alpha}_1\beta_2+\hat{\alpha}_2\beta_1-1}]=1,u^3=\ell^{2}
\rangle.
\end{equation}
Note that, using $\ell^{2}=u^3$, the relation $[u,\ell^{\hat{\alpha}_1\beta_2+\hat{\alpha}_2\beta_1-1}]=1$ can be either 
eliminated or replaced by $[u,\ell]=1$ depending on the parity of $\hat{\alpha}_1\beta_2+\hat{\alpha}_2\beta_1$.
In addition, $\ell$ can also be eliminated using $\ell^{\alpha_1 \alpha_2+\alpha_3 }=1$ and $u^3=\ell^{2}$ in case $\alpha_1 \alpha_2+\alpha_3 $ is odd.
In particular, if $\hat{\alpha}_1\beta_2+\hat{\alpha}_2\beta_1$ is even or $\alpha_1 \alpha_2+\alpha_3 $ is odd, then~\eqref{eq:pres1} becomes 
an abelian group. Otherwise, one obtains the presentation~\eqref{eq:pres-odd}.

It is immediate to verify that $\hat{\alpha}_1\beta_2+\hat{\alpha}_2\beta_1$ is odd and $\alpha_1 \alpha_2+\alpha_3 $ even if and only if 
$\alpha_1 \alpha_2 \alpha_3 \beta_1 \beta_2$ is odd, which ends the proof.
\end{proof}

\begin{cor}\label{cor:derived}
The derived subgroup $F$ of $\pi_1(\PP^2_\omega\setminus\tilde{\mathcal{C}})$ (in the non-abelian case)
is the direct product of $\ZZ/(\frac{\alpha_1 \alpha_2+\alpha_3 }{2})$ and a free group of rank~$2$. The characteristic 
polynomial of the action of the monodromy on $F/F'\otimes_{\ZZ}\CC$ is $t^2-t+1$.
\end{cor}

\subsection{A family of Zariski pairs of irreducible weighted projective curves.}

Summarizing the previous section, let $\omega=(\alpha_1 ,\alpha_2,\alpha_3)$ be pairwise coprime positive integers, and $\beta_1, \beta_2$ such 
that $\alpha_1\beta_1+\alpha_2\beta_2=\alpha_1 \alpha_2+\alpha_3 $. Consider $\mathcal{C}$ a smooth projective cubic and $\Phi_1$ (resp.~$\Phi_2$) 
the weighted Cremona transformation from $\mathbb{P}^2_\omega$ to $\mathbb{P}^2$ with respect to three tangent lines 
to $\mathcal{C}$ at aligned (resp.~non-aligned) inflection points. Let us denote by $\tilde\Phi_i^*(\mathcal{C})$ the strict transform of $\cC$
by the Cremona transformation~$\Phi_i$.

\begin{thm}\label{thm:ZP}
Under the conditions above, if $\alpha_1\alpha_2\alpha_3\beta_1\beta_2$ is odd then $(\tilde\Phi_1^*(\mathcal{C}),\tilde\Phi_2^*(\mathcal{C})\!)$
is a Zariski pair of irreducible weighted projective curves of degree $3(\alpha_1 \alpha_2+\alpha_3 )$ in~$\mathbb{P}^2_\omega$.
\end{thm}

\begin{proof}
Since both $\Phi_i$, $i=1,2$ are birational and $\mathcal{C}$ is irreducible, then $\tilde\Phi_i^*(\mathcal{C})$, $i=1,2$ 
are both irreducible as well. 
Also, the singularities of $\tilde\Phi_i^*(\mathcal{C})$ are determined locally by the singularities of the union of  
$\mathcal{C}$ and the lines used for the Cremona transformation $\Phi_i$. Hence, $\tilde\Phi_1^*(\mathcal{C})$ and 
$\tilde\Phi_2^*(\mathcal{C})$ have the same combinatorics. Finally, if $\alpha_1 \alpha_2 \alpha_3 \beta_1 \beta_2$ is odd, then 
by Proposition~\ref{prop:group_cubic} and Corollary~\ref{cor:abelian} the fundamental groups of their complements
are not isomorphic. This ends the proof.
\end{proof}

\subsection{Cyclic covers and their irregularity \`a la Esnault-Viehweg.}\label{sec:EsnaultViehweg}

The purpose of this section is to prove Theorem~\ref{thm:ZP} via a generalization of the Alexander polynomial method, 
that is, the calculation of invariants associated with cyclic covers of the weighted projective plane ramified along 
the curves. In particular, we will calculate the dimension of the eigenspaces of the homology in degree~$1$ of the cover with respect to the 
action of the deck transformation. This approach was originally used by Zariski~\cite{Zariski-irregularity} for sextics
with six cusps in the projective plane. Later on, Libgober~\cite{Libgober-alexander} and Esnault~\cite{es:82} made 
significant progress in this direction for cyclic covers and projective plane. Also Esnault-Viehweg~\cite{Esnault-Viehweg82}
gave the tools that allowed the first author in~\cite{Artal94}, Sabbah~\cite{Sabbah-Alexander}, and 
Loeser-Vaquie~\cite{Loeser-Vaquie-Alexander} to find descriptions of the irregularity of cyclic covers.
This approach was extended by Libgober~\cite{Libgober-characteristic} for abelian covers.
The approach presented here is a generalization of Esnault-Viehweg's and was developed by the authors for cyclic 
covers of surfaces with abelian quotient singularities and $\Q$-resolutions (or partial resolutions) in~\cite{ACM19}.

Let $\rho:X\to\PP^2_\omega$ be the cyclic cover of $\PP^2_\omega$ ramified along a reduced curve $\cC$ of degree~$d$. 
Consider $\check{X}=\rho^{-1}(\PP^2_\omega\setminus(\cC\cup\sing\PP^2_\omega))$ the unramified part of the cover
and let $\sigma:\check{X}\to\check{X}$ be a generator of the monodromy of the unramified cover.

Let $\pi:Y\to\PP^2_\omega$ be a $\Q$-embedded resolution of $\cC$. For $P\in\sing\cC$, let $\Gamma_P$
be the dual graph of the exceptional divisor of $\pi$ over $P$. For any $v$ vertex of $\Gamma_P$ we will denote by 
$E_v$ the associated exceptional divisor over $P$ and by $m_v$ (resp.~$\nu_v-1$) the coefficient of $E_v$ in the divisor 
$\pi^*\cC$ (resp.~in $K_\pi$, the relative canonical divisor). 

The following result describes a method to recover the dimension of the different eigenspaces of $H^1(X,\CC)$ 
with respect to the monodromy action (or deck transformation of the cover). A more general result can be found 
in~\cite[Theorem 4.4]{ACM19} for non-reduced divisors, but we state it here for covers associated with reduced divisors.

\begin{thm}[{\cite[Theorem 4.4]{ACM19}}]\label{thm:acm19}
The dimension of the eigenspace of $\sigma^*$ acting on $H^1(X;\CC)$ for the eigenvalue $e^{\frac{2i\pi k}{d}}$, $0<k<d$, 
equals $\dim\coker\pi^{(k)}+\dim\coker\pi^{(d-k)}$ where
\[
\pi^{(k)}: H^0\left(\PP^2_{\omega},\mathcal{O}_{\PP^2_w}\left( kH+K_{\PP^2_w}\right) \right) 
\longrightarrow \bigoplus_{P \in\sing\cC}
\frac{\mathcal{O}_{\PP^2_\omega,P}\left( kH+K_{\PP^2_\omega}\right)}{\mathcal{M}_{\mathcal{C},P}^{(k)}},
\]
is naturally defined given $H$ a divisor of degree~$1$, $K_{\PP^2_\omega}$ denotes the canonical divisor, and 
$\mathcal{M}_{\mathcal{C},P}^{(k)}$ is the following $\mathcal{O}_{\PP^2_\omega,P}$-module of quasi-adjunction
$$
\mathcal{M}_{\mathcal{C},P}^{(k)}:=
\left\{ g \in\mathcal{O}_{\PP^2_\omega,P}\left( kH+K_{\PP^2_\omega}\right)
\vphantom{\frac{k m_{v}}{d}}\right.
\left|\ \mult_{E_v} \pi^* g > 
\frac{k m_{v}}{d} - \nu_v, \ \forall v \in \Gamma_P \right\}.
$$
\end{thm}

Note that the module of quasi-adjunction $\mathcal{M}_{\mathcal{C},P}^{(k)}$ is a submodule of the module of equivariant germs 
$\mathcal{O}_{\PP^2_\omega,P}(\ell)$ for some $\ell=0,...,d-1$ as defined in \S\ref{sec:quotient}, namely, $\ell$ is the local 
class of the divisor $kH+K_{\PP^2_\omega}$ at~$P$. 
Our purpose will be to calculate $\dim\coker\pi^{(k)}+\dim\coker\pi^{(d-k)}$ for certain $k$ 
and $d=3(\alpha_1 \alpha_2+\alpha_3 )$ for the $d$-cyclic cover of the curves in the family presented in~\S\ref{sec:fund-groups}.

Under the conditions of Theorem~\ref{thm:ZP}, that is, $\alpha_1\alpha_2\alpha_3\beta_1\beta_2$ odd, let us consider the curve 
$\tilde \cC_\lambda:=\tilde\Phi^*_{\omega,\beta_1,\beta_2}\cC_\lambda$ as defined in \S\ref{sec:fund-groups}. 
This curve has, in general, three singular points at the vertices $P_x,P_y,P_z$. 
Recall that for $\zeta:=\exp\frac{2i\pi}{3}$ 
an easy computation given in Corollary~\ref{cor:abelian} shows that the fundamental group of $\PP^2_\omega\setminus\tilde{\cC}_\zeta$ 
is abelian and hence the first cohomology group of any cyclic cover ramified along $\tilde{\cC}_\zeta$ vanishes.

In order to understand the maps $\pi^{(k)}$ and the corresponding modules of quasi-adjunction $\mathcal{M}_{\tilde{\mathcal{C}},P}^{(k)}$
described in Theorem~\ref{thm:acm19} one needs to study the singular points of $\tilde\cC:=\tilde\cC_1$ in
$\mathbb{P}^2_\omega$. Recall that $\sing\tilde\cC\supseteq\{P_x,P_y,P_z\}$. More precisely, we will restrict
our attention to the case $\frac{k}{d}=\frac{5}{6}$.
Since $d=3(\alpha_1 \alpha_2+\alpha_3 )$, the degrees of the curves involved in $\pi^{(k)}$ is
\[
d_k=\frac{5d}{6}-(\alpha_1+\alpha_2+\alpha_3)=\frac{5\alpha_1 \alpha_2+3\alpha_3}{2}-(\alpha_1+\alpha_2).
\]

\begin{prop}\label{prop:singpz}
A $\Q$-resolution of $(\tilde\cC,P_z)$ has a dual graph with two vertices and its exceptional set is shown in 
Figure{\rm~\ref{fig:pz}}. Then $\mathcal{M}_{\tilde{\mathcal{C}},P_z}^{(k)}$, $k=\frac{5d}{6}$, is defined by
the following conditions on germs $g \in\mathcal{O}_{\PP^2_\omega,P_z}\left(d_k\right)$:
\begin{equation*}
\mult_{E_{v_z}} \pi^* g \geq
\frac{5 \alpha_1 \alpha_2 - 2(\alpha_1 +\alpha_2)}{2 \alpha_3}+\frac{1}{2},\quad 
\mult_{E_{w}} \pi^* g \geq \frac{15 \alpha_1 \alpha_2 -6(\alpha_1 +\alpha_2) }{2\alpha_3}+2.
\end{equation*}
\end{prop}

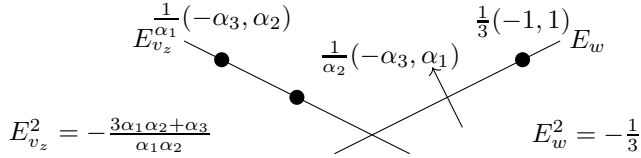
\begin{figure}[ht]
\begin{tikzpicture}
\coordinate (A) at (-2,1);
\coordinate (B) at (0,0);
\coordinate (C) at (2,1);
\draw ($1.25*(A)-.25*(B)$) node[left] {$E_{v_z}$}  --  ($-.25*(A)+1.25*(B)$);
\draw ($1.25*(C)-.25*(B)$) node[right] {$E_{w}$}  --  ($-.25*(C)+1.25*(B)$);
\draw[->] ($.5*(B)+.5*(C)+.2*(1,-2)$) -- ($.5*(B)+.5*(C)-.2*(1,-2)$);
\fill (A) circle [radius=.1cm];
\node[above=5pt] at (A) {$\frac{1}{\alpha_1}(-\alpha_3,\alpha_2)$};
\node[above right=5pt] at ($.5*(A)+.5*(B)$) {$\frac{1}{\alpha_2}(-\alpha_3,\alpha_1)$};
\fill ($.5*(A)+.5*(B)$) circle [radius=.1cm];
\fill (C) circle [radius=.1cm];
\node[above=5pt] at (C) {$\frac{1}{3}(-1,1)$};
\node[right] at ($(C)-(0,1)$) {$E_w^2=-\frac{1}{3}$};
\node[left] at ($(A)-(0,1)$) {$E_{v_z}^2=-\frac{3\alpha_1 \alpha_2+\alpha_3 }{\alpha_1 \alpha_2}$};
\end{tikzpicture}
\caption{A $\Q$-resolution of $(\tilde\cC,P_z)$}
\label{fig:pz}
\end{figure}

\begin{proof}
The result is purely local, so one can assume $\cC$ is the cubic $z x^2-(y+x)^3=0$ at the flex $[1:-1:0]$.
Then, the local equation of $\tilde\cC$ at $[0:0:1]_\omega$, regarded as $[(0,0)]\in\frac{1}{\alpha_3}(\alpha_1,\alpha_2)$, is
$x^{\beta_1} y^{\beta_2+2 \alpha_1}-(x^{\alpha_2} +y^{\alpha_1} )^3=0$. Note that 
$\alpha_1\beta_1+\alpha_2(\beta_2+2 \alpha_1)=3 \alpha_1\alpha_2 +\alpha_3>3 \alpha_1\alpha_2$.
Hence the Newton polygon of this equation is a segment of slope $-\frac{\alpha_1}{\alpha_2}$, and we perform 
an $(\alpha_1,\alpha_2)$-blowing-up. Since we start from a cyclic point one chart of this blow-up is given by
\[
(x,y)\mapsto(x^\frac{\alpha_1}{\alpha_3}, x^\frac{\alpha_2}{\alpha_3} y),
\]
i.e., the total transform is $x^{\frac{3 \alpha_1 \alpha_2}{\alpha_3}} (x y^{\beta_2+2 \alpha_1}-(1+y^{\alpha_1} )^3)=0$.
We denote this exceptional divisor as $E_{v_z}$.
Hence $m_{v_z}=\frac{3 \alpha_1 \alpha_2}{\alpha_3}$ and after a change of coordinates the strict transform (through
a smooth ambient point) has equation $x-y^3=0$. 

One can check that the multiplicity of the relative canonical divisor is $\nu_{v_z}=\frac{\alpha_1+\alpha_2}{\alpha_3}$. 
To complete the resolution, we perform a $(3,1)$-blow up, producing a new component $E_w$
for which $m_w=3\left(\frac{3\alpha_1\alpha_2}{\alpha_3}+1\right)$ and $\nu_w=3\frac{\alpha_1+\alpha_2}{\alpha_3}+1$.

By definition, the module of quasi-adjunction $\mathcal{M}_{\tilde\cC,P_z}^{(k)}$ is a submodule of 
\[
\mathcal{O}_z(d_k):=\mathcal{O}_{\PP^2_\omega,P_z}\left(d_k\right), \quad d_k=\frac{5 \alpha_1\alpha_2+3\alpha_3}{2}-(\alpha_1+\alpha_2).
\]
given by the germs $g\in\mathcal{O}_z(d_k)$ satisfying
\begin{equation}
\label{eq:mult}
\begin{aligned}
\mult_{E_{v_z}} \pi^* g & > 
\frac{k m_{v_z}}{d} - \nu_{v_z}=
\frac{5 \alpha_1\alpha_2 - 2(\alpha_1+\alpha_2)}{2\alpha_3},\\ 
\mult_{E_{w}} \pi^* g & > 
\frac{k m_{w}}{d} - \nu_{w}=
\frac{15 \alpha_1\alpha_2 +3\alpha_3-6(\alpha_1+\alpha_2) }{2\alpha_3}.
\end{aligned}
\end{equation}

Finally, note that the class of $g$ imposes extra conditions, namely, if $H=\supp(h)$, 
$h\in \mathcal{O}_z(1)$,
then $\mult_{E_{v}} \pi^* \left(\frac{g}{h^{d_k}}\right)$ must be an integer for $v\in \{v_z,w\}$. 
Using~\eqref{eq:mult} we can write
$\mult_{E_{v_z}} \pi^* g=\frac{5 \alpha_1\alpha_2 - 2(\alpha_1+\alpha_2)}{2\alpha_3}+\varepsilon_{v_z}$, for some 
$\varepsilon_{v_z}\in \mathbb{Q}_{> 0}$. Hence,
$$
\mult_{E_{v_z}} \pi^* \left(\frac{g}{h^{d_k}}\right)=\frac{5 \alpha_1\alpha_2 - 2(\alpha_1+\alpha_2)}{2\alpha_3}+\varepsilon_{v_z}-\frac{d_k}{\alpha_3}=
\varepsilon_{v_z}-\frac{3}{2}\alpha_3\in \mathbb{Z}.
$$
This implies $\varepsilon_{v_z}=\frac{1}{2}+n_{v_z}$, $n_{v_z}\in \mathbb{Z}_{\geq 0}$.
Analogously for $v=w$ one obtains
$$
\mult_{E_{w}} \pi^* \left(\frac{g}{h^{d_k}}\right)=\frac{15 \alpha_1\alpha_2 +3\alpha_3-6(\alpha_1+\alpha_2) }{2\alpha_3}+\varepsilon_{w}-3\frac{d_k}{\alpha_3}=
\varepsilon_{w}+\frac{3}{2}\in \mathbb{Z},
$$
which implies $\varepsilon_{w}=\frac{1}{2}+n_{w}$, $n_{w}\in \mathbb{Z}_{\geq 0}$ and this ends the proof.
\end{proof}

\begin{prop}\label{prop:singpx}
A $\Q$-resolution of $(\tilde{\mathcal{C}},P_x)$ is obtained with one weighted blow-up.
Then $\mathcal{M}_{\tilde{\mathcal{C}},P_x}^{(k)}$, $k=\frac{5d}{6}$, is defined by the following condition on germs
$g \in\mathcal{O}_{\PP^2_\omega,P_x}\left(d_k\right)$:
\begin{equation*}
\mult_{E_{v_x}} \pi^* g \geq
\frac{3}{\gcd(3,\alpha_1)}\cdot\frac{\alpha_1+3\beta_2-2}{2\alpha_1}+1
\end{equation*}
\end{prop}

\begin{proof}
We follow the same ideas as in the proof of Proposition~\ref{prop:singpz}.
Locally we work with the cubic $x y^2-(y+z)^3=0$ (this cubic has a flex at $[0:1:-1]$).
Then, the local equation of $\tilde{\mathcal{C}}$ at $[1:0:0]_\omega$, regarded as 
$[(0,0)]\in\frac{1}{\alpha_1}(\alpha_2,\alpha_3)=\frac{1}{\alpha_1}(1,\beta_2)$, is
$y^{\alpha_1}  z^{3}-(z+y^{\beta_2})^3=0$. We can change the coordinates (not affecting the action)
where the equation becomes $y^{\alpha_1}  (z-y^{\beta_2})^{3}-z^3=0$.
In these new coordinates the Newton polygon is non-degenerated and 
the singularity is resolved with a blowing-up with exceptional component $E_{v_x}$. Its weight is 
$(3,\alpha_1+3\beta_2)$ if $\gcd(3,\alpha_1)=1$ and $\left(1,\frac{\alpha_1}{3}+\beta_2\right)$ otherwise.

The invariants are
\begin{equation*}
m_{v_x}=3\frac{\alpha_1+3\beta_2}{\alpha_1},\quad \nu_{v_x}=\frac{\alpha_1+3\beta_2+3}{\alpha_1}.
\end{equation*}
Let us compute the quasi-adjunction module $\mathcal{M}_{\cD,P_x}^{(k)}$, as a submodule of 
$\mathcal{O}_x(\bar d_k):=\mathcal{O}_{\PP^2_\omega,P_x}\left(\bar{d}_k\right)$, where $\bar{d}_k$ is such that 
$\alpha_2\bar{d}_k\equiv d_k \bmod \alpha_1$, which implies that $\bar{d}_k\equiv \frac{\alpha_1+3\beta_2-2}{2}$.
The condition for a germ $g\in\mathcal{O}_x(\bar{d}_k)$ to be in $\mathcal{M}_{\cD,P_x}^{(k)}$ is:
\[
\mult_{E_{v_x}} \pi^* g >
3\frac{\alpha_1+3\beta_2-2}{2\alpha_1}.
\]
As above, the restriction given by $g\in \mathcal{O}_x(\bar{d}_k)$ leads to
$$
\mult_{E_{v_x}}\left(\pi^*\frac{g}{h^{\bar{d}_k}}\right)=3\frac{\alpha_1+3\beta_2-2}{2\alpha_1}+\varepsilon_{v_x}-3\frac{\bar{d}_k}{\alpha_1}=
\varepsilon_{v_x}\in \mathbb{Z}.
$$
Hence, $\varepsilon_{v_x}\in \mathbb{Z}_{> 0}$.
\end{proof}

\begin{prop}\label{lem:ker}
Let $g(x,y,z)$ be a weighted homogeneous polynomial in $\ker\pi^{(k)}$ with 
$\deg_\omega g=\frac{5(\alpha_1 \alpha_2+\alpha_3 )-2(\alpha_1+\alpha_2+\alpha_3)}{2}$. 

Then, there is a weighted homogeneous polynomial $f$, $\deg_\omega f=\alpha_1 \alpha_2+\alpha_3 $,
such that $g(x,y,z)=x^{\frac{1}{2}(\alpha_2+\beta_1-2)} y^{\frac{1}{2}(\alpha_1+\beta_2-2)}f(x,y,z)$
and
\begin{gather*}
\mult_{E_{v_z}}\pi^* f(x,y,1)\geq 
\frac{\alpha_1\alpha_2}{\alpha_3},\qquad 
\mult_{E_{w}} \pi^* f(x,y,1) \geq
\frac{3 \alpha_1\alpha_2}{\alpha_3} + \frac{1}{2},\\
\mult_{E_{v_x}} \pi^* f(1,y,z) \geq
\frac{3\beta_2}{\gcd(3,\alpha_1)\alpha_1}+1,\
\mult_{E_{v_y}} \pi^* f(x,1,z) \geq
\frac{3\beta_1}{\gcd(3,\alpha_2)\alpha_2}+1.
\end{gather*}
\end{prop}

\begin{proof}
The exponent of $x^n$ as a factor of $g$ is given by the maximal value $n\in\ZZ_{>0}$, 
such that the divisor $\supp(g)-n Y$ is effective. 
Using the generalization of Noether's multiplicity Theorem in this context,
see~\cite[Theorem~4.3(4)]{AMO-Intersection}, 
and Proposition~\ref{prop:singpz} one obtains 
\[
\begin{aligned}
(\supp(g)\cdot Y)_{P_z}&\geq \frac{(\mult_{E_{v_z}}\pi^*g)\cdot (\mult_{E_{v_z}}\pi^*y)\alpha_3}{\alpha_1\alpha_2}\\
& \geq \frac{(5\alpha_1\alpha_2-2(\alpha_1+\alpha_2))}{2\alpha_1\alpha_3}+\frac{1}{2\alpha_1}=\frac{1}{2\alpha_1\alpha_3}\left(5\alpha_1\alpha_2-2\alpha_1-2\alpha_2+\alpha_3\right).
\end{aligned}
\]
Hence,
$$
((\supp(g)-n Y)\cdot Y)_{P_z}\geq \frac{1}{2\alpha_1\alpha_3}\left(5\alpha_1\alpha_2-2\alpha_1-2\alpha_2(1+n)+\alpha_3\right).
$$
Analogously, at $P_x$ one can use Proposition~\ref{prop:singpx} to obtain
$$
\begin{aligned}
(\supp(g)\cdot Y)_{P_x} & \geq 
\left(3\frac{\alpha_1+3\beta_2-2}{2\alpha_1}+1\right)\frac{1}{\alpha_1+3\beta_2}=\frac{5\alpha_1+9\beta_2-6}{2\alpha_1(\alpha_1+3\beta_2)},
\end{aligned}
$$
regardless of the value of $\gcd(3,\alpha_1)$. Hence,
$$
\begin{aligned}
((\supp(g)-n Y)\cdot Y)_{P_x} & \geq 
\frac{5\alpha_1+9\beta_2-6-6n}{2\alpha_1(\alpha_1+3\beta_2)}.
\end{aligned}
$$
Then a global computation of the intersection multiplicity can be bounded by
$$
\begin{aligned}
((\supp(g)-n Y)\cdot  Y)_{\mathbb{P}^2_\omega} \geq 
((\supp(g)-n Y)\cdot Y)_{P_z}+((\supp(g)-n Y)\cdot Y)_{P_x}\\
\geq\ \frac{1}{2\alpha_1\alpha_3}\left(5\alpha_1\alpha_2-2\alpha_1-2\alpha_2(1+n)+\alpha_3\right)+\frac{5\alpha_1+9\beta_2-6-6n}{2\alpha_1(\alpha_1+3\beta_2)}\\
=\ \frac{1}{2\alpha_1\alpha_3}\left(5(\alpha_1 \alpha_2+\alpha_3 )-2(\alpha_1+\alpha_2+\alpha_3)-2n \alpha_2\right)-\frac{1}{\alpha_1}+\frac{5\alpha_1+9\beta_2-6-6n}{2\alpha_1(\alpha_1+3\beta_2)}\\
=\ \frac{\deg_\omega(\supp(g)-n Y)\cdot\deg_\omega(Y)}{\alpha_1\alpha_2\alpha_3}+3\frac{\alpha_1+\beta_2-2-2n}{2\alpha_1(\alpha_1+3\beta_2)}.
\end{aligned}
$$

By B\'ezout's Theorem for weighted projective planes, $n= \frac{1}{2}(\alpha_1+\beta_2-2)$. 
The same calculation applies to the divisor $X$. 
This shows that $g=x^{m}y^{n}f(x,y,z)$, $m=\frac{1}{2}(\alpha_2+\beta_1-2)$ where 
$$
\begin{aligned}
\deg(f)&=\frac{1}{2}\left(5(\alpha_1 \alpha_2+\alpha_3 )-2(\alpha_1+\alpha_2+\alpha_3)-2n \alpha_2-2m \alpha_1\right)\\
&=\frac{1}{2}\left(5(\alpha_1 \alpha_2+\alpha_3 )-2(\alpha_1+\alpha_2+\alpha_3)-(\alpha_1+\beta_2-2) \alpha_2-\!(\alpha_2\!+\!\beta_1\!-2) \alpha_1\right)\\
&=\alpha_1 \alpha_2+\alpha_3 .
\end{aligned}
$$
The last equality follows from $\alpha_1\beta_1+\alpha_2\beta_2=\alpha_1 \alpha_2+\alpha_3 $.

The last part follows immediately from Propositions~\ref{prop:singpz} and \ref{prop:singpx} and the additivity 
properties of the multiplicity.
\end{proof}

The local algebraic information obtained in this section will help us effectively study the morphism $\pi^{(k)}$ described
in Theorem~\ref{thm:acm19}. Let us use the notation introduced before Theorem~\ref{thm:ZP} and at the beginning of this 
section, let us also denote by $X_1$ (resp.~$X_2$) the cyclic cover of $\mathbb{P}^2_\omega$ of order $d=3(\alpha_1 \alpha_2+\alpha_3 )$ 
ramified along $\tilde\Phi_1^*(\mathcal{C})$ (resp.~$\tilde\Phi_2^*(\mathcal{C})$). 
Finally, denote by $L^{(k)}_i$ the invariant part of $H^1(X_i,\mathcal{O}_{X_i})$ with respect to 
the action of the monodromy by multiplication by $\exp{\frac{2\pi ik}{d}}$. Likewise, we denote by $\pi_i^{(k)}$ the 
map described in Theorem~\ref{thm:acm19} for the curve $\tilde\Phi_i^*(\mathcal{C})$. 
The discussion above shows the following.

\begin{prop}\label{prop:dimL}
If the product $\alpha_1\alpha_2\alpha_3\beta_1\beta_2$ is odd and $\frac{k}{d}=\frac{5}{6}$, 
then $\dim \ker \pi_1^{(k)}=0$ and $\dim \ker \pi_2^{(k)}=1$.
\end{prop}

\begin{proof}
By Proposition~\ref{lem:ker}, the image by $\Phi_i$ of $\supp(f)$ is a line passing through the three flexes. 
The existence of this line for $\tilde\Phi_2^*(\mathcal{C})$ but not for $\tilde\Phi_1^*(\mathcal{C})$ ends the proof.
\end{proof}

The machinery developed in this section allows one to give an alternative proof of Theorem~\ref{thm:ZP} which 
is independent of fundamental group calculations.

\begin{proof}[{Proof of Theorem~\ref{thm:ZP}}]
Since the curves
$\tilde\Phi_1^*(\mathcal{C})$ and $\tilde\Phi_2^*(\mathcal{C})$ have the same combinatorics 
and the same local type of singularities, the target space for $\pi_1^{(k)}$ and $\pi_2^{(k)}$ are the same. Therefore 
Proposition~\ref{prop:dimL} implies $\dim \coker \pi_1^{(k)}= 1+\dim \coker \pi_2^{(k)}$ for $\frac{k}{d}=\frac{5}{6}$. 
By Theorem~\ref{thm:acm19}, $\dim \coker \pi_i^{(k)}=\dim L_i^{(k)}$ is a birational invariant of $X_i$ and thus $X_1\not\cong X_2$, 
which implies that the fundamental groups of $\mathbb{P}^2_\omega\setminus \tilde\Phi_1^*(\mathcal{C})$ and 
$\mathbb{P}^2_\omega\setminus \tilde\Phi_2^*(\mathcal{C})$ are not isomorphic and thus $(\tilde\Phi_1^*(\mathcal{C}),\tilde\Phi_2^*(\mathcal{C}))$
forms a Zariski pair.
\end{proof}

\section{Some rational cuspidal curves on weighted projective planes}\label{sec:ratcusp}

The study of rational cuspidal curves in $\PP^2$ is a classical subject. There is an extensive literature about them, and we recommend 
the beautiful paper~\cite{fblmn} reviewing this topic, the most relevant conjectures, and bibliography. Two outstanding conjectures 
have been solved recently by Koras and Palka: the Nagata-Coolidge conjecture~\cite{KP:17}, that is, any rational cuspidal curve can be 
transported to a line via a Cremona transformation and such curves can have at most four singular points~\cite{koras2019Complex}. 
There is a strong knowledge of such curves in $\PP^2$ which have helped for the solution of these conjectures and other important problems, 
like the semigroup conjecture in~\cite{fblmn}, which was proven in~\cite{BL}.

Only one rational cuspidal curve in $\PP^2$ possesses four cusps: a quintic curve with singular locus $\mathbb{A}_6+3\mathbb{A}_2$.
There are many of them with three singular points, see~\cite{fz} for an infinite family. The simplest one is the cuspidal quartic with 
three ordinary cusps. The standard Cremona transformation is a way to produce this curve, namely, the standard Cremona transformation of
a smooth conic with respect to three of its tangent lines produces a tricuspidal quartic. Note that the blowing-up at the vertices does 
not affect the curve, and the blowing-downs produce the three cusps.

\subsection{Rational cuspidal curves via weighted Cremona transformations}
In this section, we will study the strict transforms of the above tritangent conic using the inverse of the weighted Cremona transformations introduced 
in~\S\ref{sec:cremona}. As a first stage, let us compute their fundamental groups. As in \S\ref{sec:zp}, let us start with the arrangement 
of a smooth conic and three lines, giving a presentation which contains suitable meridians for all the components.

Let $\mathcal{C}$ be a smooth conic and let $X,Y,Z$ be three distinct tangent lines to $\mathcal{C}$. 
If the equations of the lines are $x=0,y=0,z=0$, respectively, then the equation of~$\mathcal{C}$ (up to a suitable change of coordinates) is
\[
x^2+y^2+z^2-2(y z+x z+x y)=0.
\]

The fundamental group of the complement of the smooth conic and three tangent lines is the Artin group 
of the triangle $T(4,4,2)$ (i.e.~\cite{ji-ruben-Artin}). However, for our purposes, it is more suitable 
to use a presentation with a more geometrical interpretation. We present it here for completeness, but 
its proof is immediate using the classical Zariski-van Kampen method (as in~\cite{ji-fundamental}).

\begin{prop}
The fundamental group of $\PP^2\setminus(\mathcal{C}\cup X\cup Y\cup Z)$ is isomorphic to 
\begin{equation}
\langle
c,\ell_x,\ell_y,\ell_z\ |\ 
[\ell_x,\ell_y]=[\ell_x,\ell_z]=[\ell_y^c,\ell_z]=
\ell_y c\ell_x c\ell_z=1
\rangle.
\end{equation}
The element~$c$ is a meridian of~$\mathcal{C}$, and $\ell_x,\ell_y,\ell_z$ are meridians of~$X,Y,Z$, 
respectively. Moreover, $(\ell_x,\ell_y)$ are meridians close to $[0:0:1]$, $(\ell_x,\ell_z)$ are meridians close to $[0:1:0]$,
and $(\ell_y^c=c^{-1}\ell_yc,\ell_z)$ are meridians close to $[0:0:1]$.
\end{prop}

Considering $u:=c\ell_z$, the above presentation of $\pi_1(\PP^2\setminus(\mathcal{C}\cup X\cup Y\cup Z))$ can be alternatively written as:
\begin{equation}\label{eq:grupo2bis}
\langle
u,\ell_x,\ell_y,\ell_z\ |\ 
[\ell_x,\ell_y]=[\ell_x,\ell_z]=[\ell_y^u,\ell_z]=
u\ell_y u\ell_x\ell_z^{-1} =1
\rangle.
\end{equation}

As in \S\ref{sec:zp}, fixing $\omega,\beta_1,\beta_2$, we consider the birational map $\Phi$ and we denote 
by $\tilde{\mathcal{C}}$ the strict transform of $\mathcal{C}$ by~$\Phi$.

\begin{prop}\label{prop:semidirect}
Let $d:=\gcd(\alpha_1+2\beta_2,\alpha_2+2\beta_1)$. 
Then $\pi_1(\PP^2_\omega\setminus\tilde{\mathcal{C}})$ is the semidirect product
$(\ZZ/d)A\ltimes(\ZZ/2(\alpha_1 \alpha_2+\alpha_3 ))B$ where $B A B^{-1}=A^{-1}$.
Hence the group has size $2 d(\alpha_1\alpha_2 +\alpha_3)$ and it is abelian if and only if $d=1$.
\end{prop}

\begin{rem}
Note that $d$ is odd, since $\alpha_1,\alpha_2$ cannot be simultaneously even. Moreover,
$$
\begin{aligned}
\gcd(d,\alpha_1)&=\gcd(\alpha_1,\beta_2,\alpha_2+2\beta_1)
=\gcd(\alpha_1,\alpha_2\beta_2,\alpha_2+2\beta_1)\\
&=\gcd(\alpha_1,\alpha_3,\alpha_2+2\beta_1)=1
\end{aligned}
$$
and analogously $\gcd(d,\alpha_2)=1$. Note also
\[
\gcd(d,\beta_1)=\gcd(d,\beta_1,\alpha_1+2\beta_2,\alpha_2)=1,
\]
and hence $\gcd(d,\beta_2)=1$. The following congruences can easily be checked:
\[
\alpha_1 \beta_1\equiv -2 \beta_1\beta_2\equiv \alpha_2 \beta_2 \mod d.
\]
Moreover,
$$
\begin{aligned}
\gcd(d,\alpha_1 \alpha_2+\alpha_3)&=\gcd(d, \alpha_1 \alpha_2+\alpha_3 , \alpha_2(\alpha_1 +2 \beta_2))\\
&=\gcd(d, \alpha_1 \alpha_2+\alpha_3 , \alpha_1(\alpha_2 -2 \beta_1))\\
&=\gcd(d, \alpha_1 \alpha_2+\alpha_3 , 2 \alpha_1 \beta_1)=1.
\end{aligned}
$$
\end{rem}

\begin{proof}[Proof of Proposition{\rm~\ref{prop:semidirect}}]
The presentation of  $\pi_1(\PP^2_\omega\setminus\tilde{\mathcal{C}})$ is obtained from~\eqref{eq:grupo2bis}
by adding the relations which \emph{kill} the meridians of the exceptional divisors $E_x,E_y,E_z$:
\begin{equation}\label{eq:rels_2}
\ell_x\ell_y=\ell_x^{\alpha_1}\ell_z^{\beta_2}=
u^{-1}\ell_y^{\alpha_2}u\ell_z^{\beta_1}=1.
\end{equation}
Let us first check that the abelianization of this quotient is $\ZZ/2(\alpha_1 \alpha_2+\alpha_3 )$. We will denote by $[\bullet]$
the class of $\bullet$ in the abelianization. Note that $[\ell_z]=[u]^2$; moreover, using B{\'e}zout's identity 
and the equations in~\eqref{eq:rels_2}, both $[\ell_x]$ and $[\ell_y]$ can be expressed in terms of $[\ell_z]$. Hence,
the abelianization is cyclic. A presentation matrix in terms of the generators $[\ell_y],[u]$ is given by
\[
\begin{pmatrix}
\alpha_2& 2\beta_1\\
-\alpha_1& 2\beta_2
\end{pmatrix}
\]
whose determinant, $2(\alpha_1 \alpha_2+\alpha_3 )$, is the size of the abelianization. 

Let us study now the group itself. Note first that $\ell_y$ can be eliminated from~\eqref{eq:rels_2} as $\ell_y=\ell_x^{-1}$.
Let us check that $u^2$ is central. The last relation in~\eqref{eq:grupo2bis} can be written as
\[
u^2=\ell_z\ell_x^{-1}(u^{-1}\ell_x u).
\]
We deduce that $u^2$ commutes with $\ell_z$, since it commutes with each factor; hence $u^2$ also commutes with 
$u\ell_z u^{-1}$. Also note that
\[
\ell_x^{\alpha_1}=\ell_z^{-\beta_2},\quad \ell_x^{\alpha_2} =u\ell_z^{\beta_1}u^{-1}\Longrightarrow
\ell_x=\ell_z^{-\hat{\alpha}_1\beta_2} u\ell_z^{-\hat{\alpha}_2\beta_1}u^{-1}.
\]
Then $u^2$ commutes with $\ell_x$ and it is central. 

Using the last relation in~\eqref{eq:grupo2bis} $\ell_z$ can also be eliminated as $\ell_z=u\ell_x^{-1} u\ell_x$.
The presentation of the group becomes:
\begin{equation*}
\langle
u,\ell_x\ |\ 
[\ell_x,u\ell_x^{-1} u]=[\ell_x^u,u\ell_x^{-1} u\ell_x]=
\ell_x^{\alpha_1}(u\ell_x^{-1} u\ell_x)^{\beta_2}=
u^{-1}\ell_x^{-q} u(u\ell_x^{-1} u\ell_x)^{\beta_1}=1
\rangle
\end{equation*}
which can be further simplified using the centrality of $u^2$:
\begin{equation*}
\langle
u,\ell_x | 
[\ell_x,u^2]\!=\![\ell_x,u\ell_x u^{-1}]\!=\!
u^{2\beta_2}\ell_x^{\alpha_1+\beta_2}(\!u\ell_x u^{-1}\!)^{-\beta_2}\!=\!
u^{2\beta_1} (\!u \ell_x u^{-1}\!)^{-(\alpha_2+\beta_1\!)}\ell_x^{\beta_1}\!=\!1
\rangle
\end{equation*}

The map $\pi_1(\PP^2_\omega\setminus\tilde{\mathcal{C}})\to\ZZ/2$
given by $u\mapsto 1$ and $\ell_x,\ell_z\mapsto 0$ is well defined.
A presentation of its kernel~$K$ is obtained using Reidemeister-Schreier method.
The generators are $X_0:=\ell_x$, $X_1:=u\ell_x u^{-1}$,
and $U:=u^2$. The first two relations imply that the group is abelian; the other relations yield:
\begin{align*}
u^{2\beta_2}\ell_x^{\alpha_1+\beta_2}(u\ell_x u^{-1})^{-\beta_2}=1 &\Longrightarrow 
U^{\beta_2}X_1^{\alpha_1+\beta_2} X_0 ^{-\beta_2}=U^{\beta_2}X_0^{\alpha_1+\beta_2} X_1 ^{-\beta_2}=1, \\
u^{2\beta_1} (u \ell_x u^{-1})^{-(\alpha_2+\beta_1)}\ell_x^{\beta_1}=1 &\Longrightarrow
U^{\beta_1} X_1^{-(\alpha_2+\beta_1)}X_0^{\beta_1}=
U^{\beta_1} X_0^{-(\alpha_2+\beta_1)}X_1^{\beta_1}=1.
\end{align*}
Let us express these relations in a matrix (the rows represent the relations and the columns stand for 
the generators). These relations become (recall $d=\gcd(\alpha_1+2\beta_2,\alpha_2+2\beta_1)$):
\begin{gather*}
\begin{pmatrix}
d&-d&0\\
-\beta_2& \alpha_1 +\beta_2 &\beta_2\\
\beta_1& -(\alpha_2+\beta_1)&\beta_1\\
\hline
X_0& X_1 & U
\end{pmatrix}\sim
\begin{pmatrix}
d&0&0\\
-\beta_2& \alpha_1  &\beta_2\\
\beta_1 & -\alpha_2 &\beta_1\\
\hline
X_0 X_1^{-1}& X_1 & U
\end{pmatrix}.
\end{gather*}
The determinant of this matrix is $d(\alpha_1 \alpha_2+\alpha_3 )$ which is the size of~$K$; the greatest common divisor of the $2$-minors
divides $d$ and $2\alpha_1\beta_1$, i.e., it is~$1$. Hence, the group~$K$ is cyclic of order $d(\alpha_1 \alpha_2+\alpha_3 )$.
One also has that $D:=X_0 X_1^{-1}$ is of order~$d$. Considering the product of the second relation to the power~$\alpha_2$
and the third relation to the power~$\alpha_1$ one obtains
\[
1=D^{\alpha_1\beta_1-\alpha_2\beta_2} U^{\alpha_1 \alpha_2+\alpha_3 }= U^{\alpha_1 \alpha_2+\alpha_3 }.
\]
From the order in the abelianization, we deduce that $U$ is of order $\alpha_1 \alpha_2+\alpha_3 $. Hence $K$ is the direct product of
the cyclic group of order~$d$ generated by $D$ and the cyclic group of order~$\alpha_1 \alpha_2+\alpha_3 $ generated by $U$. The conjugation
by $u$ satisfies $u U u^{-1}=U$ and $u D u^{-1}=D^{-1}$. The result follows.
\end{proof}

Embedded $\Q$-resolutions of the singularities of these curves can be computed as in 
Propositions~\ref{prop:singpz} and~\ref{prop:singpx}, see Figure~\ref{fig:rcc}.

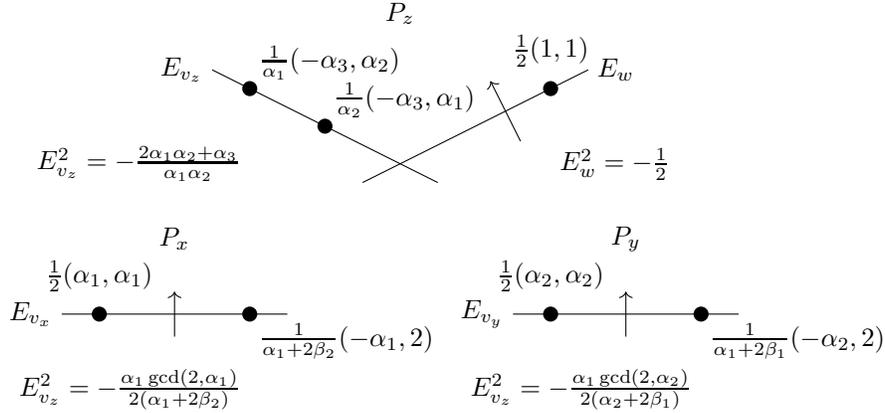
\begin{figure}[ht]
\begin{center}
\begin{tikzpicture}
\coordinate (A) at (-2,1);
\coordinate (B) at (0,0);
\coordinate (C) at (2,1);
\node at (0,2) {$P_z$};
 
\draw ($1.25*(A)-.25*(B)$) node[left] {$E_{v_z}$}  --  ($-.25*(A)+1.25*(B)$);
\draw ($1.25*(C)-.25*(B)$) node[right] {$E_{w}$}  --  ($-.25*(C)+1.25*(B)$);
\draw[->] ($.5*(B)+.7*(C)+.2*(1,-2)$) -- ($.5*(B)+.7*(C)-.2*(1,-2)$);

\fill (A) circle [radius=.1cm];
\node[above right=0pt] at (A) {$\frac{1}{\alpha_1}(-\alpha_3,\alpha_2)$};
\node[above right=0pt] at ($.5*(A)+.5*(B)$) {$\frac{1}{\alpha_2}(-\alpha_3,\alpha_1)$};
\fill ($.5*(A)+.5*(B)$) circle [radius=.1cm];
\fill (C) circle [radius=.1cm];
\node[above=5pt] at (C) {$\frac{1}{2}(1,1)$};

\node[right] at ($(C)-(0,1)$) {$E_w^2=-\frac{1}{2}$};
\node[left] at ($(A)-(0,1)$) {$E_{v_z}^2=-\frac{2\alpha_1 \alpha_2+\alpha_3 }{\alpha_1\alpha_2}$};

\begin{scope}[xshift=-3cm,yshift=-2cm]
\coordinate (A) at (-1,0);
\coordinate (B) at (1,0);
\coordinate (C) at (2,1);

\node at (0,1) {$P_x$};
 
\draw ($1.25*(A)-.25*(B)$) node[left] {$E_{v_x}$}  --  ($-.25*(A)+1.25*(B)$);

\draw[->] ($.5*(A)+.5*(B)-.3*(0,1)$) -- ($.5*(A)+.5*(B)+.3*(0,1)$);

\fill (A) circle [radius=.1cm];
\node[above=5pt] at (A) {$\frac{1}{2}(\alpha_1,\alpha_1)$};
\node[below right=0pt] at ($(B)$) {$\frac{1}{\alpha_1+2\beta_2}(-\alpha_1,2)$};
\fill ($(B)$) circle [radius=.1cm];

\node[left] at ($(A)-(-2,1)$) {$E_{v_z}^2=-\frac{\alpha_1\gcd(2,\alpha_1)}{2(\alpha_1+2\beta_2)}$};
\end{scope}

\begin{scope}[xshift=3cm,yshift=-2cm]
\coordinate (A) at (-1,0);
\coordinate (B) at (1,0);
\coordinate (C) at (2,1);

\node at (0,1) {$P_y$};
 
\draw ($1.25*(A)-.25*(B)$) node[left] {$E_{v_y}$}  --  ($-.25*(A)+1.25*(B)$);

\draw[->] ($.5*(A)+.5*(B)-.3*(0,1)$) -- ($.5*(A)+.5*(B)+.3*(0,1)$);

\fill (A) circle [radius=.1cm];
\node[above=5pt] at (A) {$\frac{1}{2}(\alpha_2,\alpha_2)$};
\node[below right=0pt] at ($(B)$) {$\frac{1}{\alpha_1+2\beta_1}(-\alpha_2,2)$};
\fill ($(B)$) circle [radius=.1cm];

\node[left] at ($(A)-(-2,1)$) {$E_{v_z}^2=-\frac{\alpha_1\gcd(2,\alpha_2)}{2(\alpha_2+2\beta_1)}$};
\end{scope}
\end{tikzpicture}
\caption{Singularities of the rational cuspidal curves}
\label{fig:rcc}
\end{center}
\end{figure}

\subsection{Rational cuspidal curves via weighted Kummer covers}

There is another simple way to produce rational cuspidal curves in weighted projective planes from this arrangement
of curves. It is quite simple but it will be shown to be useful in the upcoming sections. Let $d_1,d_2,d_3$ be 
pairwise coprime integers and let $\omega:=(e_1,e_2,e_3)$, where $e_i:=d_j d_k$, $\{i,j,k\}=\{1,2,3\}$. Following \S\ref{sec:WPR} 
note that $\eta=(1,1,1)$ and thus there is an isomorphism $\mathbb{P}^2_\omega\to \mathbb{P}^2$ given by 
$[x:y:z]_\omega\mapsto [x^{d_1}:y^{d_2}:z^{d_3}]$. 
This map gives a geometrical interpretation to the group
\[
G_{d_1,d_2,d_3}:=\pi_1(\PP^2\setminus(\mathcal{C}\cup X\cup Y\cup Z))/\langle\ell_x^{d_1}=\ell_y^{d_2}=\ell_z^{d_3}=1\rangle.
\]
as the orbifold group of $\PP^2_\omega$ with respect to the curve $\mathcal{C}\cup X\cup Y\cup Z$ and index
$e(\mathcal{C})=0,n(X)=d_1,n(Y)=d_2,n(Z)=d_3$ as defined in~\cite{ACM-orbifoldgroups}. Briefly, if $X$ is a smooth projective surface,
$D=D_1\cup...\cup D_s$ is a normal crossing union of smooth hypersurfaces, and $n_i:=n(D_i)\in \mathbb{Z}_{\geq 0}$, then one can 
define the \emph{orbifold fundamental group} $\pi_1^{\orb}(X)$ of $X$ with respect to $D$ with indices $n_i$ as the quotient 
of the group $\pi_1(X\setminus D)$ by the normal subgroup generated by $\gamma_i^{n_i}$, where $\gamma_i$ is a meridian of $D_i$. 
If $D$ do not have normal crossings, then one resolves to a normal crossing divisor by blowing up points and defines the index at an 
exceptional divisor as the least common multiple of the indices of the components passing through the point (in this context we
set~$\lcm(0,n)=0$).

\begin{prop}\label{prop:group_orbifold}
The abelianization of the group $G_{d_1,d_2,d_3}$ is $\ZZ/2d_1d_2d_3$. 
The group is abelian if and only if at least one of $d_1,d_2,d_3$ equals~$1$.
\end{prop}

\begin{proof}
The computation for the abelianization is straightforward. Assume $d_3=1$, then 
$$
\begin{aligned}
G_{d_1,d_2,1}&=\langle\ell_x,u\mid [\ell_x,\underbrace{u\ell_x}_{v^{-1}} u]=1, \ell_x^{d_1}=(u\ell_x u)^{d_2}=1\rangle\\
&=\langle\ell_x,v\mid [\ell_x,v^2]=1, \ell_x^{d_1}=\ell_x^{d_2} v^{2d_2}=1\rangle.
\end{aligned}
$$
Using B\'ezout's identity we can express $\ell_x$ in terms of $v$ and the result follows.

For the case $(d_1,d_2,d_3)\neq (1,1,1)$, let us consider the double cover of $\PP^2$ ramified along~$\cC$. It is $\PP^1\times\PP^1$
and the preimage~$\tilde{\cC}$ of $\cC$ is the diagonal; the three lines are transformed in pairs of vertical-horizontal
lines $X_{\pm},Y_{\pm},Z_{\pm}$ intersecting $\tilde{\cC}$, see Figure~\ref{fig:doublecover}.

\begin{figure}[ht]
\begin{center}
\begin{tikzpicture}
\draw (-2,-2) -- (2,2) node[right] {$\tilde{\mathcal{C}}$} ;
\draw (-2,-1) node[above] {$X_-$}-- (2,-1);
\draw (-2,0) node[above] {$Y_-$} -- (2,0);
\draw (-2,1) node[above] {$Z_-$} -- (2,1);
\draw (-1,-2) node[right] {$X_+$}  --  (-1,2) ;
\draw (0,-2)  node[right] {$Y_+$}    --  (0,2) ;
\draw (1,-2)  node[right] {$Z_+$} -- (1,2);
\end{tikzpicture}
\caption{Double cover ramified along $\cC$}
\label{fig:doublecover}
\end{center}
\end{figure}
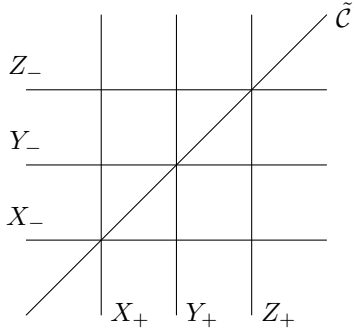

Note that the fundamental group of $\PP^1\times\PP^1\setminus(\tilde{\cC}\cup X_+\cup X_-\cup Y_+\cup Y_-\cup Z_+\cup Z_-)$
is isomorphic to the fundamental group of the projective complement of Ceva's arrangement.
The kernel $K_{d_1,d_2,d_3}$ of the map $G_{d_1,d_2,d_3}\to\ZZ/2$, $u\mapsto 1$ and the images of the other generators vanish,
it is an index~2 subgroup of $G_{d_1,d_2,d_3}$ and it is an orbifold fundamental group for the
above configuration in $\PP^1\times\PP^1$; the projection on each component induces orbifold morphisms onto $\PP^1_{d_1,d_2,d_3}$,
where $\PP^1_{d_1,d_2,d_3}$ is an orbifold modeled on $\PP^1$, with three quotient points of order $d_1,d_2,d_3$, respectively,
and its orbifold fundamental group is isomorphic to 
$\langle \mu_1,\mu_2,\mu_3\mid \mu_1^{d_1}=\mu_2^{d_2}=\mu_3^{d_3}=\mu_3\mu_2\mu_1=1\rangle$,
a triangle group. The combination of the two projections induces an 
epimorphism of $K_{d_1,d_2,d_3}$ onto $\pi_1^{\orb}(\PP^1_{d_1,d_2,d_3})$. 
The triangle is hyperbolic if $\{d_1,d_2,d_3\}\neq\{2,3,5\}$ and hence its group is infinite. 
If $\{d_1,d_2,d_3\}=\{2,3,5\}$, the triangle group is the alternating group $A_5$, with cardinal~$60$.
Using \texttt{GAP}~\cite{GAP4} we can compute the intersection of the kernels of the two projections (which is a subgroup
of index~$3600$); its abelianization is $\ZZ^{59}$ and the result follows. The computation can be checked
in \url{https://github.com/enriqueartal/AnOrbifoldFundamentalGroup} using
\texttt{Sagemath}~\cite{sage89} and \texttt{Binder}~\cite{binder}.
\end{proof}

\subsection{A rational cuspidal curve with four cusps}

We present in this section a nice example in $\PP^2_{(1,1,2)}$, which is a rational curve of degree~$6$ with $4$ ordinary cusps
--\,incidentally, note that $4$ is the maximal number of singular points a rational cuspidal curve can have in $\mathbb{P}^2$.
Let us explain how to construct it via Cremona transformations. Let us start
with a tricuspidal quartic $\cC_0$; this curve, dual of the nodal cubic, has a bitangent line~$L$. Let $P_0\in\cC_0\cap L$;
its blown-up produces a ruled surface~$\Sigma_1$, where the negative section~$E$ is the exceptional component and $\cC_1$,
the strict transform of $\cC_0$ has three cusps and one tangent fiber.
Let us consider the Nagata transformation at $\cC_1\cap E$; the result is $\Sigma_2$ and the blow-down of the negative
section produces $\PP^2_{(1,1,2)}$. The strict transform $\cC$ of $\cC_1$ is the desired curve.

\begin{prop}
The fundamental group $\pi_1(\PP^2_{(1,1,2)}\setminus(\cC\cup\{P_z\}))$ has a presentation
\[
\langle
s,t,u\mid
s t s= t s t, s u s=u s u, t u t=u t u, (s t u)^2=1
\rangle.
\]
\end{prop}

\begin{proof}
Following the construction it is the fundamental group of $\Sigma_2\setminus(\cC_2\cup E_2)$ where $\cC_2$ is
the strict transform of $\cC$ and $E_2$ is the negative section. The Zariski-van Kampen method applied to the ruling yields
the result.
\end{proof}

\subsection{Milnor fibers}
We have presented in \S\ref{sec:zp} and in \S\ref{sec:ratcusp} several examples of irreducible 
quasi-projective curves such that their (maybe orbifold) fundamental groups are non-abelian. 
As a consequence their cones are quasi-homogeneous non-isolated surface singularities
in~$\CC^3$ with non simply-connected Milnor fibers.

If $F\in\CC[x,y,z]$ is a homogeneous polynomial of degree~$d$, an important topological invariant is its Milnor fiber. The Milnor fiber of a homogeneous singularity is 
a fiber of $F:\CC^3\setminus F^{-1}(0)\to\CC^*$, say $F^{-1}(1)$. The restriction to $F^{-1}(1)$ of $F$ of the natural map 
$\CC^3\setminus\{0\}\to\PP^2$ is
a $d$-cyclic cover onto the complement of the tangent cone~$\cC_d$ in $\PP^2$, defined by an epimorphism
$\pi_1(\PP^2\setminus\cC_d)\to\ZZ/d$. 

If $\omega$ is a weight and 
$F\in\CC[x,y,z]$ is an $\omega$-quasi-homogeneous polynomial of $\omega$-degree~$d$, its Milnor fiber 
$F=1$ can also be recovered as a $d$-cyclic orbifold cover of $\PP^2_\omega\setminus\cC_d$ (the complement
of the \emph{tangent $\omega$-quasi-cone})
defined by an epimorphism
$\pi_1^{\orb}(\PP^2_\omega\setminus\cC_d)\to\ZZ/d$. If the elements of $\omega$ are pairwise coprime and the vertices
are in~$\cC_d$, then the notions of $\pi_1$ and $\pi_1^{\orb}$ coincide. If it is not the case, the notion
of orbifold fundamental groups apply.

The curves obtained via the Cremona transformation provide homogeneous singularities whose topology is not complicated
and such that the Milnor fiber has non-trivial fundamental group. The following result is a direct consequence of
Proposition~\ref{prop:semidirect}.

\begin{prop}
Let
$\omega=(\alpha_1, \alpha_2,\alpha_3)$ be pairwise coprime and let $\beta_1,\beta_2$ be such that 
$\alpha_1 \alpha_2+\alpha_3 =\alpha_1\beta_1+\alpha_2\beta_2$.
Let $S:=\{F_{\omega,\beta_1,\beta_2}(x,y,z)=0\}$, where
\[
F_{\omega,\beta_1,\beta_2}(x,y,z)=y^{2\alpha_1} z^2+ x^{2\alpha_2} z^2+ x^{2\beta_1} y^{2\beta_2}-
2 z(x^{\alpha_2}  y^{\alpha_1}  z+x^{\alpha_2+\beta_1} y^{\alpha_1} + x^{\alpha_2}  y^{\alpha_1+\beta_2})
\]
defines a $\omega$-homogeneous singularity. Then, the fundamental group of the Milnor fiber of $S$ 
is cyclic of order~$\gcd(\alpha_1+2\beta_2,\alpha_2+2\beta_1)$.
\end{prop}

More complicated fundamental groups can be obtained by choosing the orbifold variant. 
Let $(d_1,d_2,d_3)$ be a triple of pairwise coprime integers, $d_i>1$, let 
$\omega=(d_2 d_3,d_1 d_3,d_1 d_2)$ be a weight, and let 
\[
F_{d_1, d_2,d_3}(x,y,z)=x^{2 d_1}+y^{2d_2}+z^{2d_3}-2(x^{d_1} y^{d_2}+x^{d_1} z^{d_3}+ y^{d_2} z^{d_3}).
\]
As a direct consequence of Proposition~\ref{prop:group_orbifold} we obtain the following result.

\begin{prop}
The fundamental group of the Milnor fiber of $\{F_{d_1, d_2,d_3}=0\}$ is infinite and non-abelian. 
\end{prop}

\section{Weighted L{\^e}-Yomdin surface singularities}\label{sec:LY}

In this section we study the relationship between (weighted) projective plane curves and normal surface singularities
whose link is a rational (or integral) homology sphere. 

\subsection{The determinant of a normal surface singularity}\label{sec:determinant}
Let $(S,0)$ be a germ of normal surface singularity and let $K$ be its link.
It is well known that $K$ is a graph manifold whose plumbing decorated graph is the dual graph $\Gamma$ of a simple normal crossing
resolution. Each vertex $v$ of $\Gamma$ is decorated with two numbers $(g_v,e_v)$, where $g_v$ is the genus of the corresponding
irreducible component $E_v$ and $e_v$ is its self-intersection. Let $A$ be the intersection matrix of the graph; recall
that $A$ is negative definite. In a natural way, $A$ is also the presentation matrix of an abelian group yielding the following classical result.

\begin{prop}
The free part of $H_1(K;\ZZ)$ has rank $2\sum_v g_v+\rk H_1(\Gamma;\ZZ)$. The torsion part is isomorphic 
to $\coker A$ and, in particular, its cardinality is~$\det (-A)$.
\end{prop}

As a direct consequence of this, the determinant $\det (-A)$ does not depend on the resolution.
This justifies the definition of the determinant of a normal surface singularity.

\begin{defn}
The \emph{determinant} $\det S$ of a normal surface singularity~$S$ is defined as $\det (-A)$, where $A$
is the intersection matrix of any resolution of~$S$.
\end{defn}

As a consequence, one has the following combinatorial criteria to detect rational (resp.~integral) homology 
sphere singularities, that is, surface singularities whose link is a rational (resp.~integral) homology 
sphere.

\begin{cor}
\label{cor:QHS}
The surface singularity $S$ is a rational (resp.~integral) homology sphere if and only if all $g_v$'s vanish 
and $\Gamma$ is a tree (resp.~and $\det S=1$). 
\end{cor}

\subsection{Superisolated and L{\^e}-Yomdin singularities} 
In~\cite{fblmn}, the authors relate hypersurface singularities whose link is a rational homology sphere 
with rational cuspidal curves using superisolated singularities. In our search for more examples of surface
singularities whose link is a rational (or integral) homology sphere, a generalization of this method will
be discussed here.  For the sake of completeness we present a classical result.

\begin{defn}
Let $(S,0)\subset(\CC^3,0)$ be the germ of a hypersurface singularity with equation $F=f_d+f_{d+k}+\dots$, 
where the previous decomposition is the decomposition in homogeneous parts. Assume $f_d\neq 0$, $k>0$. Let
$\cC_m:=V_{\PP}(f_m)$ denote the projective zero locus in~$\PP^2$ of the homogeneous polynomial $f_m$. We
say that $S$ is a \emph{L{\^e}-Yomdin} singularity if $\sing(\cC_d)\cap \cC_{d+k}=\emptyset$. If $k=1$,
$S$ is called a \emph{superisolated} singularity.
\end{defn}

Superisolated singularities were introduced by Luengo~\cite{lu:87a}: they can be resolved by one blow-up.
In~\cite{LMN-LinksSuperisolated}, the authors show that the link of a superisolated singularity is a rational 
homology sphere if and only if all the irreducible components of $\cC_d$ are cuspidal rational and if the curve 
is reducible they only intersect at one point. Besides the smooth case, no other one provides an integral 
homology sphere as can be deduced from the following result in~\cite{lu:87a}. We reproduce the proof since it 
will be generalized for other classes of singularities.

\begin{prop}[{\cite{lu:87a}}]\label{prop:luengo}
Let $S$ be a superisolated singularity with tangent cone $\cC_d$ of degree~$d$. 
Let $\Pi:\widehat{\CC}^3\to\CC^3$ be the blow-up of $0\in S\subset \CC^3$ and 
$\pi:\hat{S}\to S$ the restriction of $\Pi$ to the strict transform of $S$. 
If $E\cong\PP^2$ is the exceptional divisor of $\Pi$, then the exceptional divisor of $\pi$ is $\cC_d=E\cap\hat{S}$.

Moreover, if $\cC_{d,1},\dots,\cC_{d,\s}$ denote the irreducible components of $\cC_d$ and $\dd_i:=\deg\cC_{d,i}$, 
then
\[
(\cC_{d,i}\cdot\cC_{d,i})_{\hat{S}}=-\dd_i(d-\dd_i+1),\quad (\cC_{d,i}\cdot\cC_{d,j})_{\hat{S},P}=
(\cC_{d,i}\cdot\cC_{d,j})_{\PP^2,P}, i\neq j, P\in\sing\cC_d.
\]
\end{prop}

\begin{proof}
Let us assume that $[0:0:1]\in\sing\cC_d$. We can fix the usual chart of the blowing-up. Assume that 
$S=\{F=0\}$, where $F=f_d+f_{d+1}+\dots$; in the chart $(x,y,z)\mapsto(x z,y z,z)$ and
$E=\{z=0\}$, $\hat{S}=\{f_d(x,y,1)+z(f_{d+1}(x,y,1)+\dots)=0\}$, i.e., $\cC_d=\{z=f_d(x,y,1)=0\}$.
In the neighborhood of $P$, $(E,\cC_d)$ and $(\hat{S},\cC_d)$ are isomorphic. We deduce that 
for $i\neq j$, $(\cC_{d,i}\cdot\cC_{d,j})_{\hat{S},P}=(\cC_{d,i}\cdot\cC_{d,j})_{\PP^2,P}$.

The surfaces $E$ and $\hat{S}$ are \emph{generically} transversal, namely outside $\sing\cC_d$. 
The Euler class $e(E)=-L$, where $L$ is a line in~$E$. Then,
\[
(\cC_d\cdot\cC_{d,i})_{\hat{S}}=(e(E)\cdot\cC_{d,i})_{\PP^2}=-\dd_i.
\]
Also
\begin{gather*}
(\cC_d\cdot\cC_{d,i})_{\hat{S}}=(\cC_{d,i}\cdot\cC_{d,i})_{\hat{S}}
+\sum_{j\neq i}(\cC_{d,j}\cdot\cC_{d,i})_{\hat{S}}=
(\cC_{d,i}\cdot\cC_{d,i})_{\hat{S}}+\sum_{j\neq i}(\cC_{d,j}\cdot\cC_{d,i})_{\PP^2}=\\
(\cC_{d,i}\cdot\cC_{d,i})_{\hat{S}}+\dd_i(d-\dd_i),
\end{gather*}
and the result follows.
\end{proof}

Although $\pi$ is not necessarily a resolution with normal crossings, $\det S$ can be recovered from 
it using its intersection matrix; it is a classical result which will follow from a later proposition.

\begin{cor}\label{cor:sis}
If $S$ is as above, then
$$\det S=(d+1)^{\s-1}\cdot \dd_1\cdot\ldots\cdot \dd_\s.$$ 
In particular, if $\cC_d$ is irreducible, then $\det S=d$.
\end{cor}

\begin{proof}
By Proposition~\ref{prop:luengo}, the diagonal terms of the intersection matrix for~$\pi$ equal $-\dd_i(d-\dd_i+1)$ 
and the non-diagonal terms are $\dd_i\cdot \dd_j$. Replacing the first row by the sum of all rows, one obtains 
$-(\dd_1,\dots,\dd_\s)$. If we add the new first row multiplied by $\dd_i$ times the $i^{\text{th}}$-row ($i>1$), 
all the non-diagonal terms vanish and the diagonal term becomes $-\dd_i(d+1)$.
\end{proof}

For L{\^e}-Yomdin singularities we follow the same strategy. If $S=\{F=0\}$ with $F=f_d+f_{d+k}+\dots$, and 
we keep the notation above, the main difference is that $\hat{S}$ is no longer smooth, in general. If $P\in\sing\cC_d$, then
the local equation of $\hat{S}$ at~$P$ is $z^k-f(x,y)=0$ where $f(x,y)=0$ is the local equation of $\cC_d$ at~$P$.
Intersection theory can be used also in normal surfaces, see~\cite{mum:61,Fulton-Intersection} for definitions 
and~\cite{AMO-Intersection} for useful tips. 
As the following result shows the intersection form of a partial resolution is also useful.

\begin{lem}\label{lem:detS}
Let  $(S,0)$ be a normal surface singularity and let $\pi:(X,D)\to(S,0)$ be a proper birational morphism which is
an isomorphism outside~$D=\pi^{-1}(0)$ on the normal surface $X$. Let $A$ be the intersection matrix for~$D$. Then,
\[
\det S=\det(-A)\prod_{P\in D}\det(X,P).
\]
\end{lem}

\begin{proof}
Note first that the product in the formula is finite since only a finite number of singular points may arise.
Let $\sigma:(Y,E)\to(X,D)$ be a resolution of the singularities of~$X$. Let $B$ the intersection matrix of~$E$.
Instead of expressing this matrix in terms of the irreducible components of~$E$, we replace the strict transforms
of the components of~$D$ by their total transforms. 

Then, $B$ is replaced by a matrix~$\tilde{B}$, with the same determinant, which is a diagonal sum of $A$ and 
the intersection matrices of the singular points. Then,
\begin{equation*}
\det S=\det(-B)=\det(-\tilde{B})=\det(-A)\prod_{P\in\sing X}\det(X,P).
\qedhere
\end{equation*}
\end{proof}

\begin{prop}\label{prop:luengo-ly}
Let $S$ be a $k$-L{\^e}-Yomdin singularity with tangent cone $\cC_d$ of degree~$d$. 
With the notation of Proposition{\rm~\ref{prop:luengo}},
\[
(\cC_{d,i}\cdot\cC_{d,i})_{\hat{S}}=-\frac{\dd_i(d-\dd_i+k)}{k},\  
(\cC_{d,i}\cdot\cC_{d,j})_{\hat{S},P}=\frac{(\cC_{d,i}\cdot\cC_{d,j})_{\PP^2,P}}{k}, i\neq j, P\in\sing\cC_d.
\]
\end{prop}

\begin{proof}
We follow the guidelines of the proof of Proposition~\ref{prop:luengo}.
Note that it is not true any more that in the neighborhood of $P\in\sing\cC_d$ the germs $(E,\cC_d)$ and $(\hat{S},\cC_d)$ 
are isomorphic. However, the projection $\rho(x,y,z):=(x,y)$ restricts to a $k:1$ proper map $(\hat{S},\cC_d)\to(E,\cC_d)$.
Since $\pi^*(\pi_*\cC_{d,i})=k\cC_{d,i}$ we have that for $i\neq j$
\[
(\cC_{d,i}\cdot\cC_{d,j})_{\hat{S},P}=\frac{1}{k^2}(\pi^*\pi_*\cC_{d,i}\cdot\pi^*\pi_*\cC_{d,j})_{\hat{S},P}=
\frac{1}{k}(\cC_{d,i}\cdot\cC_{d,j})_{\PP^2,P}.
\]
For the self-intersections we apply the same ideas:
\begin{align*}
(\cC_d\cdot\cC_{d,i})_{\hat{S}}&=(e(E)\cdot\cC_{d,i})_{\PP^2}=-\dd_i \\
(\cC_d\cdot\cC_{d,i})_{\hat{S}}&=(\cC_{d,i}\cdot\cC_{d,i})_{\hat{S}}+\sum_{j\neq i}(\cC_{d,j}\cdot\cC_{d,i})_{\hat{S}}=
(\cC_{d,i}\cdot\cC_{d,i})_{\hat{S}}+\frac{\dd_i(d-\dd_i)}{k},
\end{align*}
and the result follows.
\end{proof} 

A similar proof to the one of Corollary~\ref{cor:sis} provides the following result.

\begin{cor}\label{cor:ly}
If $S$ is a $k$-L{\^e}-Yomdin as above, then 
$$\det S=\dd_1\cdot\ldots\cdot \dd_\s\cdot \left(\frac{d+k}{k}\right)^{\s-1}\prod_{P\in \sing\cC_d}\det S_{P,k},$$ 
where 
\[
S_{P,k}=\{z^k=f_P(x,y)\mid P\in\sing\cC_d\},
\]
and $f_P(x,y)=0$ is a local equation of $\cC_d$ at $P$.

In particular, if $\cC_d$ is smooth, then $\det S=d$.
\end{cor}

\begin{ex}
Let $S_k$ be the singularity $\{z^k=x^2+y^2\}$, then $\det S_k=k$. 
Denote by $T_k$ the singularity $\{z^k=x^2+y^3\}$, then we have
\[
\det T_k=
\begin{cases}
1&\text{ if }\gcd(k,6)=1,6\\
3&\text{ if }\gcd(k,6)=2\\
4&\text{ if }\gcd(k,6)=3.\\
\end{cases}
\]
Note that $T_k$ admits a $\Q$-resolution with only one exceptional divisor.
This divisor has positive genus (equal to one) if and only if $\gcd(k,6)=6$.
\end{ex}

We did not find in the literature a general formula for this determinant.
From the above computations and the periodicity properties of the Alexander
invariants, the following statement may be true.

\begin{cjt}
Let $C:f(x,y)=0$ be a germ of a reduced plane curve singularity, and let 
$S_k:z^k=f(x,y)$ be a cyclic germ of surface. Let $N$ be the order of the semisimple
factor of the monodromy of $C$. Then $\det S_k$
is a quasi-polynomial in $k$ of period $N$.
\end{cjt}

\begin{prop}\label{prop:ly_rhs}
A $k$-L{\^e}-Yomdin singularity with tangent cone~$\cC_d$ has as link a rational homology
sphere if and only if $\cC_d$ is a union of rational cuspidal curves with only one intersection point
and the links of the $k$-cyclic singularities associated with the singular points of $\cC_d$
have also a rational homology sphere as a link.
\end{prop}

The proof of this proposition is a direct consequence of the previous result. 
We have not proven that L{\^e}-Yomdin singularities do not provide integral homology sphere links, mainly since we do
not have a closed formula for the determinant of a cyclic singularity. Our experimentation
leads to this conjecture.

\begin{cjt}
No $k$-L{\^e}-Yomdin singularity $k>1$ has an integral homology sphere link.
\end{cjt}

\subsection{Weighted L{\^e}-Yomdin singularities}

We are going to generalize these families of singularities using weighted homogeneous curves.
We use the notation $\omega$, $\eta$, etc.~introduced in~\S\ref{sec:WPR}.
The following notion of weighted L{\^e}-Yomdin singularity was introduced in~\cite{ABLM-milnor-number}.

\begin{defn}
A hypersurface $(S,0):=\{F=0\}$ is an \emph{$(\omega,k)$-weighted L{\^e}-Yomdin singularity} if the following
holds. Let $F:=f_d+f_{d+k}+\dots$ be the decomposition in $\omega$-weighted homogeneous forms, then 
$\Jac(f_d)\cap \supp(f_{d+k})=\emptyset$.
\end{defn}

In order to relate geometrically this definition with the definition of superisolated and
L{\^e}-Yomdin singularities, let us consider the weighted blow-up $\Pi_\omega:\widehat{\CC}_\omega^3\to\CC^3$.
In \S\ref{sec:wblowups} we have described a stratification of the exceptional divisor $E_\omega\cong\PP^2_\omega\cong\PP^2_\eta$
according to the singularities of  $\widehat{\CC}_\omega^3$, see Proposition~\ref{prop:singular_locus}. 

One needs to study the two curves $\cC_d,\cC_{d+k}\subset E_\omega$.
In general, note that $f_d(x,y,z)=x^{\varepsilon_x}y^{\varepsilon_y}z^{\varepsilon_z} g(x^{d_1},y^{d_2},z^{d_3})$, 
where $\varepsilon_x,\varepsilon_y,\varepsilon_z\in\{0,1\}$ and $g$ is $\eta$-weighted homogeneous
of degree~$\frac{d-e_1\varepsilon_x-e_2\varepsilon_y-e_3\varepsilon_z}{d_1 d_2 d_3}$. 
If we see this curve in $\PP^2_\eta$ its equation is $x^{\varepsilon_x}y^{\varepsilon_y}z^{\varepsilon_z} g(x,y,z)=0$.

\begin{prop}
\label{prop:wLY}
Let $S=\{F=0\}$ be an $(\omega,k)$-weighted L{\^e}-Yomdin singularity with $\omega$-quasi-tangent cone $\cC_d=\{f_d=0\}$. 
Let $\Pi_\omega$ be the $\omega$-blow-up, $E_\omega\cong\PP^2_\omega\cong\PP^2_\eta$ is the exceptional divisor and $\hat{S}$ 
is the strict transform (and a partial resolution) of~$S$. Recall the stratification of 
$E_\omega=\mathcal{P}\cup\mathcal{L}\cup \mathcal{T}$ as given in Notation{\rm~\ref{ntc:strata}}. 
The structure of $\hat{S}$ along $P\in\cC_d=E_\omega\cap\hat{S}$ is as follows:

\begin{enumerate}[label=\rm(\arabic{enumi})]
\item\label{caso1} $P\in\mathcal{T}$.

\begin{enumerate}[label=\rm(\alph{enumii}), ref=\rm(\arabic{enumi}\alph{enumii})]

\item\label{caso1a} If $P\notin \sing\cC_d$ then $\hat{S}$ is smooth at $P$ and 
$E_\omega\pitchfork_P\hat{S}$.

\item\label{caso1b} If $P\in \sing\cC_d$  then $P\notin\cC_{d+k}$.
There are local coordinates $U,V,W$ centered at~$P$ such that $E_\omega=\{W=0\}$, $\cC_d=\{W=g(U,V)=0\}$
and $\hat{S}=\{W^k=g(U,V)\}$; in particular $\hat{S}$ is smooth at~$P$ if and only if $k=1$ (but it is not transversal to $E_\omega$).

\end{enumerate}

\item\label{caso2} $P\in\mathcal{L}_y$ (a similar statement holds for $\mathcal{L}_x,\mathcal{L}_z$).

\begin{enumerate}[label=\rm(\alph{enumii}), ref=\rm(\arabic{enumi}\alph{enumii})]
\item\label{caso2a} 
If $\cC_d$ is transversal to~$Y$ at~$P$  then $(\hat{S},P)\cong\frac{1}{d_2}(e_2,-1)$. 
In the quotient ambient space  $(\hat{\CC}_\omega^3,P)$ the situation is similar to {\rm~\ref{caso1a}}. 

\item\label{caso2b} 
If $(\cC_d,P)=(Y,P)$ then $(\hat{S},P)$ is smooth. 
In the quotient ambient space $(\widehat{\CC}_\omega^3,P)$ the situation is similar to {\rm~\ref{caso1a}}. 

\item\label{caso2c} 
If $\cC_d\not\pitchfork_P Y$, i.e. the order of $f_d(x+t,y,1)$ is $>1$ ($P=[t:0:1]$), then  $P\notin\cC_{d+k}$.
The germ $(\hat{S},P)$ is isomorphic to $z^k=f_d(x+t,y,1)$ in
the 3-fold quotient singularity $\frac{1}{d_2}(0,e_2,-1)$, where $z=0$ is the equation of~$E_\omega$.
\end{enumerate}
 
\item\label{caso3} $P=P_z$ (a similar statement holds for $P_x,P_y$).

\begin{enumerate}[label=\rm(\alph{enumii}), ref=\rm(\arabic{enumi}\alph{enumii})]
\item\label{caso3a} If $\cC_d$ is extremely quasi-smooth at~$P$ (i.e. the order of $f_d(x,y,1)$ is~$1$) 
the situation is as in{\rm~\ref{caso1a}} replacing the ambient smooth space by
the 3-fold quotient singularity $\frac{1}{e_3}(e_1,e_2,-1)$.
Let $h_1(x,y)$ be the linear part of $f_d(x,y,1)$.

\begin{enumerate}[label=\rm(\roman{enumiii}),ref=\rm(\arabic{enumi}\alph{enumii}\roman{enumiii})]
 \item If $h_1(x,y)$ is proportional to~$x$, then $(\hat{S},P)\cong\frac{1}{e_3}(e_1,-1)$.
 \item If $h_1(x,y)$ is proportional to~$y$, then $(\hat{S},P)\cong\frac{1}{e_3}(e_2,-1)$.
 \item Otherwise, $e_1\equiv e_2\bmod{e_3}$ and the above cases coincide.
\end{enumerate}

\item\label{caso3b} 
If $\cC_d$ is not extremely quasi-smooth at~$P$ (i.e., the order of $f_d(x,y,1)$ is~$>1$), then  $P\notin\cC_{d+k}$
and $d+k\equiv 0\bmod{e_3}$. 
The germ $(\hat{S},P)$ is isomorphic to $z^k=f_d(x,y,1)$ in
the 3-fold quotient singularity $\frac{1}{e_3}(e_1,e_2,-1)$, 
where $z=0$ is the equation of~$E_\omega$.
\end{enumerate}
\end{enumerate}
\end{prop}

\begin{proof}
The different parts of the statement will be particular cases of the following general situation.
Assume $P=[0:0:1]\in \cC_d=E_\omega\cap \hat S$ is a point of the strict transform $\hat S$ of 
$S$ on the exceptional divisor $E_\omega$. The total transform of $S$ is equal to 
$\hat{S}+dE_\omega$ and hence, its equation in the chart $\Psi_{\omega,3}$ is:
\[
z^d(f_d(x,y,1)+z^k\underbrace{(f_{d+k}(x,y,1)+\dots)}_{q(x,y)})=0.
\]
By hypothesis $f_d(0,0,1)=0$, let us denote by $\ell(x,y)$ the linear part of $f_d(x,y,1)$. 
The following conditions are immediate
\[
\begin{cases}
P\notin\cC_{d+k}\Longleftrightarrow q(x,y)\text{ is a unit},\\
P\in \cC_{d+k}\Longrightarrow \ell(x,y)\neq 0.
\end{cases}
\]
We will consider $(x_1,y_1,z_1)$ a change of coordinates where
\[
(x_1,y_1,z_1)=
\begin{cases}
(x,y,z q(x,y)^{\frac{1}{k}})&\text{ if } \ell(x,y)=0,\\
(\frac{1}{a}(f_d(x,y,1)+z^k q(x,y)),y,z)&\text{ if } \ell(x,y)=ax,\\
(x,\frac{1}{b}(f_d(x,y,1)+z^k q(x,y)-ax),z)&\text{ if } \ell(x,y)=ax+by, b\neq 0.\\
\end{cases}
\]
Note that the action of $\mu_{e_3}$ on $(x_1,y_1,z_1)$ reads as in $(x,y,z)$.
If $P\notin\cC_{d+k}$ then $d+k\equiv 0\bmod e_3$.

In these coordinates $E_\omega:z_1=0$ and $\cC_d:W=g(x_1,y_1)=0$, where 
\[
g(x_1,y_1)=
\begin{cases}
f_d(x_1,y_1,1)&\text{ if } \ell(x,y)=0,\\
\ell(x_1,y_1)&\text{ otherwise. }
\end{cases}
\]
The local equations for $d E_\omega+\hat{S}$ are
$z_1^d (z_1^k+g)=0$. If $\ell(x,y)\neq 0$, then both look like two surfaces in a 
quotient ambient space whose preimages in $\CC^3$ are smooth and transversal.

The case $P\in\mathcal{T}$ locally corresponds to $\omega=(1,1,1)$. Note that \ref{caso1a} implies 
$\ell(x_1,y_1)\neq 0$ whereas \ref{caso1b} implies $\ell(x_1,y_1)=0$. 
The case $P\in\mathcal{L}_y$ corresponds with the choice $\omega=(d_2,e_2,d_2)$ where \ref{caso2a} 
and \ref{caso2b} refers to $\ell\neq 0$ and \ref{caso2c} refers to $\ell=0$. 
Finally, $P=P_z$, corresponds to the choice $\omega=(e_1,e_2,e_3)$. In this case~\ref{caso3a} refers to
the different cases of $\ell\neq 0$ and \ref{caso3b} refers to $\ell=0$. 
\end{proof}

The divisor $\cC_d$ has an irreducible decomposition in $\s+\varepsilon_x+\varepsilon_y+\varepsilon_z$ components
$\varepsilon_x X+\varepsilon_y Y+\varepsilon_z Z+ \tilde{\cC}_{d}$, where $\tilde{\cC}_{d}=\sum_{i=1}^\s \cC_{d,i}$
and $\varepsilon_x,\varepsilon_y,\varepsilon_z\in\{0,1\}$. 
Recall that $d_1d_2d_3$ divides~$\deg\cC_{d,i}$ and we can write $\dd_i=d_1d_2d_3\hat{\dd}_{i}$.

\begin{def}
\label{def:ssmooth}
Consider the stratification of a weighted projective plane as above.
We call a curve in a weighted projective plane \emph{stratified smooth} if it is smooth, intersects the axes 
transversally and does not contain the vertices.
\end{def}

\begin{prop}\label{prop:luengo-wly}
Let $S$ be an $(\omega,k)$-weighted L{\^e}-Yomdin singularity with \emph{quasi-tangent cone} 
$\cC_d\subset\PP^2_\omega$ of degree~$d$. 
With the notation of Proposition{\rm~\ref{prop:luengo-ly}}:
\begin{enumerate}[label=\rm(\arabic{enumi})]
\item$(\cC_{d,i}\cdot\cC_{d,i})_{\hat{S}}=-\frac{\dd_i(d-\dd_i+k)}{k e_1 e_2 e_3}=-\frac{\hat{\dd}_{i}(d-\dd_i+k)}{k d_1d_2d_3\alpha_1 \alpha_2 \alpha_3}$.
\item\label{interx} If $\varepsilon_x=1$, then $(X\cdot X)_{\hat{S}}=-\frac{d_1^2 e_1(d-e_1+k)}{k e_1 e_2 e_3}=-\frac{d_1^2 (d-e_1+k)}{k e_2 e_3}=-\frac{d-e_1+k}{k d_2 d_3 \alpha_2 \alpha_3}$. 
Similar formulas hold for $Y$ and~$Z$.
\item\label{interij} If $i\neq j$, then $(\cC_{d,i}\cdot\cC_{d,j})_{\hat{S}}=\frac{\dd_i\dd_j}{k e_1 e_2 e_3}=\frac{\hat{\dd}_i\hat{\dd}_j}{k \alpha_1 \alpha_2 \alpha_3}$.
\item\label{interix} If $\varepsilon_x=1$ then $(\cC_{d,i}\cdot X)_{\hat{S}}=\frac{d_1\dd_i}{k  e_2 e_3}=\frac{\hat{\dd}_i}{k \alpha_2 \alpha_3}$.
Similar formulas hold for $Y$ and~$Z$.
\item\label{interxy} If $\varepsilon_x\varepsilon_y=1$
$(X\cdot Y)_{\hat{S},P_z}=\frac{d_1 d_2}{k  e_3}=\frac{1}{k  \alpha_3}$.
Similar formulas hold for the other pairs involving~$X$, $Y$, and~$Z$.
\end{enumerate}
\end{prop}

\begin{proof}
We follow the ideas in the proofs of Propositions~\ref{prop:luengo} and~\ref{prop:luengo-ly} with some modifications.
If $\omega\neq\eta$, the map $\pi_{\eta,\omega}^{-1}:\PP^2_\eta\to\PP^2_\omega$ can be considered as the \emph{identity}
where $\PP^2_\omega$ is seen as $(\PP^2_\eta)^\orb$, where $(\pi_{\eta,\omega}^{-1})^*(X)=\frac{1}{d_1} X$, 
$(\pi_{\eta,\omega}^{-1})^*(Y)=\frac{1}{d_2} Y$, and $(\pi_{\eta,\omega}^{-1})^*(Z)=\frac{1}{d_3} Z$. Analogously, the abstract 
strict transform $\hat{S}$ has a natural orbifold embedded structure $\hat{S}^\orb\subset\widehat{\CC}^3_\omega$
where the embedding $\pi:\hat{S}\to\hat{S}^\orb$ has the same properties for $X,Y,Z$ as $\pi_{\eta,\omega}^{-1}$
whenever $X,Y,Z$ are contained in $\cC_d$.

The divisor
$e(E_\omega)$ in $E_\omega\equiv\PP^2_\omega$ has degree~$1$ and B\'ezout's Theorem for the $\omega$-projective
plane states that the sum of the intersection numbers of two divisors equals the product of the degrees
divided by~$e_1e_2e_3$. Hence, we obtain the same formulas as in Proposition~\ref{prop:luengo-ly} with two differences:
$e_1e_2e_3$ appears in the denominator and all the intersection numbers are considered in $\hat{S}^\orb$.

When we consider the intersection numbers in $\hat{S}$, when $X$ appears ($\varepsilon_x=1$), the formulas must be multiplied
by~$d_1$. A similar argument holds for $Y,Z$.
\end{proof}

\begin{cor}\label{cor:wly}
If $S$ is an $(\omega,k)$-weighted L{\^e}-Yomdin as above and $A$ is the intersection matrix of the blowing-up, then 

$$\det S=d_1^{2\varepsilon_x}\cdot d_2^{2\varepsilon_y}\cdot d_3^{2\varepsilon_z}\cdot \dd_1\cdot\ldots\cdot \dd_\s 
\cdot \left(\frac{d+k}{k e_1e_2e_3}\right)^{\s-1} \prod_{P\in \hat{S}} \det(\hat{S}_{k,P}),$$ 
where $\hat{S}_{k,P}$ is the surface singularity at $P$ as described in Proposition{\rm~\ref{prop:wLY}}.

In particular, if $\cC_d$ is stratified smooth, then $\det S=\frac{d}{e_1e_2e_3}$.
\end{cor}

\section{Normal surface singularities with rational homology sphere links}

In this section we will use the results and strategies presented in section~\ref{sec:LY} in order 
to exhibit examples of weighted L{\^e}-Yomdin singularities whose links are rational homology spheres,
generalizing the strategy in~\cite{fblmn}. We will be using Proposition~\ref{prop:ly_rhs} in the 
context of weighted L{\^e}-Yomdin singularities.

\subsection{Brieskorn-Pham singularities}
\label{ex:bp}
We will interpret these singularities as L{\^e}-Yomdin singularities and study their $\Q$-resolution graph. 
Consider $\omega_0=(n_1,n_2,n_3)$ and the Brieskorn-Pham singularity $S=\{F_{\omega_0}=x^{n_1}+y^{n_2}+z^{n_3}=0\}\subset (\CC^3,0)$,
where $n_1,n_2,n_3$ are not assumed to be coprime.

Denote by $e:=\gcd\omega_0$, and $\alpha_{k}:=\frac{1}{e}\gcd(n_i,n_j)$,
where $\{i,j,k\}=\{1,2,3\}$. Note that $d_i:=\frac{n_i}{e\alpha_j \alpha_k}\in \ZZ_{>0}$ are pairwise coprime.
If 
$$
\omega=(e_1,e_2,e_3):=\frac{1}{e^2\alpha_{1}\alpha_{2}\alpha_{3}}
{\left(n_2n_3,n_1n_3,n_1n_2\right)}=
(\alpha_1 d_2 d_3,\alpha_2 d_1 d_3,\alpha_3 d_1 d_2),
$$
then $F_{\omega_0}(x,y,z)$ is an $\omega$-weighted homogeneous polynomial of degree $d:=\frac{n_1 n_2 n_3}{e^2\alpha_{1}\alpha_{2}\alpha_{3}}$
and hence $S$ can be viewed as an $(\omega,k)$-weighted L{\^e}-Yomdin singularity for any $k\geq 1$. 
Following the general construction, $f_d=F_{\omega_0}(x,y,z)=g(x^{d_1},y^{d_2},z^{d_3})=0$ 
can be considered a curve in $\PP^2_\eta\cong \PP^2_\omega$ for $\eta=(\alpha_1,\alpha_2,\alpha_3)$ of $\eta$-degree 
$d_\eta=e\alpha_1\alpha_2\alpha_3$ given by the equation $g(x,y,z)=x^{e\alpha_2\alpha_3}+x^{e\alpha_1\alpha_3}+z^{e\alpha_1\alpha_2}=0$. Its genus is 
$$
\frac{d_\eta(d_\eta-|\eta|)}{2\alpha_1\alpha_2\alpha_3}+1=\frac{e^2\alpha_1\alpha_2\alpha_3-
e(\alpha_1+\alpha_2+\alpha_3)+2}{2}.
$$
Since the curve $\cC_d$ is transversal to the axes we obtain that the exceptional locus of $\hat{S}$ has (in the 
intersection with the axes) $e \alpha_i$ cyclic points of order $d_i$.
The determinant of the singularity is
\[
\frac{d}{(d_1d_2d_3)^2 (\alpha_1\alpha_2\alpha_3)} \left(d_1^{\alpha_1} d_2^{\alpha_2} d_3^{\alpha_3}\right)^e=
e d_1^{e\alpha_1-1} d_2^{e\alpha_2-1} d_3^{e\alpha_3-1}.
\]
As a consequence of this discussion one obtains the following.

\begin{prop}
\label{prop:BPham}
The Brieskorn-Pham singularity $S=\{F_{\omega_0}=x^{n_1}+y^{n_2}+z^{n_3}=0\}\subset (\CC^3,0)$ is a rational homology
sphere singularity if and only if either $\alpha_1=\alpha_2=\alpha_3=1, e=2$ or $\alpha_i=\alpha_j=e=1$ for some $i\neq j$.

Moreover, it is an integral homology sphere if and only if the exponents are pairwise coprime.
\end{prop}

\subsection[Examples from Cremona and Kummer]{Examples coming from Cremona transformations and Kummer covers}
\label{sec:RHSCremona}
The purpose of this section is to provide more candidates to surface singularities with rational homology sphere links by applying the 
techniques used in~\S\ref{sec:ratcusp}. In particular, we will start with the strict transforms of the conic by the Cremona transformations.

In order to do so one needs $\omega:=(\alpha_1,\alpha_2,\alpha_3)$ pairwise coprime, 
and $\beta_1,\beta_2\in\ZZ_{>0}$ such that $\alpha_1 \alpha_2+\alpha_3 =\alpha_1\beta_1+\alpha_2\beta_2$. The weighted homogeneous polynomial
\[
\begin{aligned}
f_\omega(x,y,z)&=f(y^{\alpha_1} z,x^{\alpha_2} z,x^{\beta_1}y^{\beta_2})\\
&=y^{2\alpha_1} z^2+ x^{2\alpha_2} z^2+ x^{2\beta_1} y^{2\beta_2}-2 z\left(x^{\alpha_2}  y^{\alpha_1}  z+ x^{\beta_1} 
y^{\alpha_1+\beta_2}+ x^{\alpha_2+\beta_1} y^{\beta_2}\right)
\end{aligned}
\]
has $\omega$-degree $2(\alpha_1 \alpha_2+\alpha_3 )$ and defines a rational curve in $\PP^2_\omega$ which is smooth outside the vertices. 
Assume for simplicity that $\alpha_1 \alpha_2+\alpha_3 <\alpha_1\alpha_2\alpha_3$. Hence for any generic quasi-homogeneous polynomial 
$g_\omega(x,y,z)$ of degree $2\alpha_1\alpha_2\alpha_3$, $F:=f_\omega+g_\omega$ defines an $(\omega,k)$-weighted L{\^e}-Yomdin
singularity, for $k=2(\alpha_1\alpha_2\alpha_3-\alpha_1\alpha_2-\alpha_3)$. A partial resolution of this singularity has an exceptional 
locus which is a rational curve with three singular points. In most cases the link of this singularity is a rational homology sphere. 
For simplicity, we will prove it in a special case.

\begin{prop}
With the previous notation, take $\alpha_1=1$, $\beta_1=\alpha_3$, and $\beta_2=1$.
Then, for any $\alpha_2,\alpha_3>1$ satisfying $\gcd(3,k)=\gcd(3,\alpha_2\alpha_3-\alpha_2-\alpha_3)=1$ and a generic $g_\omega$, the equation
$\{F=f_\omega+g_\omega=0\}\subset \mathbb{C}^3$ 
defines a surface singularity with a rational homology sphere link.
\end{prop}

\begin{proof}
We study the strict transform of this singularity at $P_x,P_y,P_z$ after an $\omega$-weighted blow-up. 
At $P_x$, the ambient space is smooth and the strict transform has equation 
\[
0=x^k+y^{2} z^2+ z^2+ y^{2}-2 y z\left( z+  y+ 1\right)=x^k+(y_1+z)^{2} z^2+ y_1^{2}-2 (y_1+z) z\left( 2z+  y_1\right)
\]
if $y_1=y-z$. This is topologically equivalent to $0=x^k+y^2+z^3$. 
Since $\gcd(3, \alpha_2\alpha_3-\alpha_2-\alpha_3)=1$, by Proposition~\ref{prop:BPham} the link of this singularity is a rational homology sphere.

At $P_y$, the ambient space is $\frac{1}{\alpha_2}(1,-1,\alpha_3)$ and the strict transform has equation 
$$
\begin{aligned}
f_\omega(x,y,z)&=y^k+z^2+ x^{2\alpha_2} z^2+ x^{2\alpha_3}-2 z\left(x^{\alpha_2}  z+ x^{\alpha_3} + x^{\alpha_2+\alpha_3} \right)\\
&=y^k+z_1^2-2 x^{\alpha_2+2\alpha_3}+\dots
\end{aligned}
$$
if $z_1=z-x^{\alpha_3}$. This change of variable is compatible with the action; this equation defines a singularity in $\CC^3$  whose 
link is a rational homology sphere, and so is the case in the quotient manifold.

By symmetry arguments, the same happens for $P_z$. Hence, $F$ defines a singularity  whose link is a rational homology sphere.
\end{proof}

Let us use the orbifold approach. Given $(d_1,d_2,d_3)$ pairwise coprime consider $\omega:=(d_2 d_3,d_1 d_3, d_1 d_2)$; the normalized $\eta$ is $(1,1,1)$
and the isomorphism $\mathbb{P}^2_\omega\to \mathbb{P}^2$ is given by $[x:y:z]_\omega\mapsto [x^{d_1}:y^{d_2}:z^{d_3}]$, see~\eqref{eq:P2isom}. 
This isomorphism can be seen as a \emph{weighted Kummer cover} 
and the homogeneous polynomial
\[
f_\omega(x,y,z)=f(x^{d_1},y^{d_2},z^{d_3})=x^{2d_1} + y^{2d_2}+ z^{2d_3} -2 \left(y^{d_2} z^{d_3}+ x^{d_1} z^{d_3}+ x^{d_1} y^{d_2}\right)
\]
of $\omega$-degree $2 d_1d_2d_3$, which defines a rational curve in $\PP^2_\omega\cong\PP^2$ and it is tangent to the axes. 
In most cases the link of this singularity is a rational homology sphere. Let us study a special case.

\begin{prop}
For any generic quasi-homogeneous polynomial $g_\omega(x,y,z)$ of degree $3d_1d_2d_3$, and $d_1,d_2,d_3$ odd numbers
$\{F:=f_\omega+g_\omega=0\}\subset \mathbb{C}^3$ defines a surface singularity with a rational homology sphere link.
\end{prop}

\begin{proof}
Note that $\{F=0\}$  defines an $(\omega,k)$-weighted L{\^e}-Yomdin singularity, for $k=d_1d_2d_3$. 
A partial resolution of this singularity has an exceptional locus which is a rational curve with three singular points 
(corresponding to the tangencies). In most cases the link of this singularity is a rational homology sphere. 

By symmetry reasons we study only the strict transform of this singularity at the tangency point with $Y$ after an 
$\omega$-weighted blow-up. After a change of coordinates the local equation of $F$ is
\begin{gather*}
0=z^{d_1d_2d_3}+ (x+1)^{2d_1} + y^{2d_2}+ 1 -2 \left(y^{d_2} + (x+1)^{d_1} + (x+1)^{d_1} y^{d_2}\right)\\=
z^{d_1d_2d_3}+{d_1}^2 x^2 -2 y^{d_2}+\dots
\end{gather*}
and the ambient space is $\frac{1}{d_2}(0,d_1d_3,-1)$. Since $d_1,d_2,d_3$ are odd numbers, by Proposition~\ref{prop:BPham} this 
equation defines a singularity in $\CC^3$ whose link is a rational homology sphere, and so is the case in the quotient manifold.
\end{proof}

\subsection{New examples of integral homology sphere surface singularities}
\label{sec:newexamples}
Note that the only integral homology spheres we have found are well known in the literature, which justifies this subsection.

Following ideas of the third named author, Veys, and Vos, we present an infinite family of normal 
surface singularities which are complete intersection in~$\mathbb{C}^4$ and whose links are integral 
homology spheres. The examples given here can be generalized to any dimension.

Let $n_0,n_1,n_2,n_3\in \mathbb{Z}_{>0}$ and $b_{20},b_{21},b_{30},b_{31},b_{32}\in \mathbb{Z}_{\geq 0}$.
Consider $S$ the surface singularity 
in $(\mathbb{C}^4,0)$ defined by
\begin{equation}
\label{eq:SJorge}
S=\{f_1 + f_2 = f_2 + f_3 = 0\}\subset (\mathbb{C}^4,0), \text{ where }
\begin{cases}
f_1 = x_1^{n_1} - x_0^{n_0}, \\
f_2 = x_2^{n_2} - x_0^{b_{20}} x_1^{b_{21}}, \\
f_3 = x_3^{n_3} - x_0^{b_{30}} x_1^{b_{31}} x_2^{b_{32}}.
\end{cases}
\end{equation}
Note that the family of surfaces $S$ contains the Brieskorn-Pham surface singularities, for instance 
when $n_0=b_{30}=1$, $b_{20}=b_{21}=b_{31}=b_{32}=0$.

The purpose of this section is to show when the link of $S$ is a rational homology sphere 
as well as to characterize when it is integral. The idea is to resolve $S$ with $\mathbb{Q}$-normal 
crossings and apply Lemma~\ref{lem:detS} to compute $\det S$. In order to do so
we consider the Cartier divisors of $S$ defined by $Y = \{ f_1 = 0 \}$ and $H_i = \{x_i = 0\}$, $i=0,1,2$.
This family was recently studied in Vos' PhD thesis in a more general context and we just briefly
discuss here the construction of the partial resolution obtained in~\cite[section 5]{MVV19}.

\begin{thm}
Let $S\subset (\mathbb{C}^4,0)$ be the surface singularity defined above. Assume  
$n_0,n_1,n_2,n_3\in \mathbb{Z}_{>0}$ are pairwise coprime, then $S$ is a rational homology sphere.
Moreover, in that case $S$ is an integral homology sphere singularity if and only if 
$m:=\gcd(n_3,b_{20}n_1+b_{21}n_0)=1$.
\end{thm}

\begin{proof}
Let $\pi_1: \widehat{\mathbb{C}}^4 \to \mathbb{C}^4$ be the weighted blow-up at the origin of $\mathbb{C}^4$
with weights $w_1 = (\frac{n}{n_0},\frac{n}{n_1},\frac{n}{n_2},\frac{n}{n_3})$ where $n = n_0 n_1 n_2 n_3$.
The exceptional divisor of $\pi_1$ is the weighted projective variety $E_1 = \mathbb{P}^3_{w_1}$.
The assumption on the integers $n_i$, $i=0,..,3$ being pairwise coprime implies that the exceptional divisor 
$\mathcal{E}_1$ of the restriction $\varphi_1 = \pi_1|_{\hat{S}}: \hat{S} \to S$ is a rational irreducible 
curve which contains three singular points of $\hat{S}$, namely 
$Q_0 = \hat{H}_0 \cap \mathcal{E}_1$, $Q_1 = \hat{H}_1 \cap \mathcal{E}_1$,
and $P_1 = \hat{H}_2 \cap \mathcal{E}_1 = \hat{Y} \cap \mathcal{E}_1$,
see figure~\ref{fig:1st-resolution-S}. The local type of the singularities at $Q_0$ and $Q_1$ are given by
\begin{equation*}
Q_0: \left\{\begin{aligned}
& \hat{S} = \frac{1}{n_0}(n_1 n_2 n_3, -1) \\
& \mathcal{E}_1: \ x_1 = 0, \quad \hat{H}_0^{\text{red}}: \ x_0 = 0,
\end{aligned}\right. \qquad
Q_1: \left\{\begin{aligned}
& \hat{S} = \frac{1}{n_1} (-1, n_0 n_2 n_3) \\
& \mathcal{E}_1: \ x_0 = 0, \quad \hat{H}_1^{\text{red}}: \ x_1 = 0.
\end{aligned}\right. \\
\end{equation*}
Around $P_1$ the surface $\hat{S}$ can be described inside $\frac{1}{n_2 n_3}(-1, n_0 n_1 n_3, n_0 n_1 n_2)$
as the set of zeros of $x_2^{n_2} - x_0^{b_2'} + x_3^{n_3}  - x_0^{b_3'} x_2^{b_{32}}
+ (x_2^{n_2}-x_0^{b_2'}) R_2'(x_0,x_2)$ where
$b_i' = b_{i,0} \frac{n}{n_0} + \cdots + b_{i,i-1} \frac{n}{n_{i-1}}-n$, $i=2,3$, and $R_i'(0,x_2)=0$.
Since the monomial with higher order will not play any role in the resolution of $S$, roughly speaking
the situation at $P_1$ with variables $[(x_0,x_2,x_3)]$ can be thought of as
\begin{equation}\label{eq:P1}
P_1: \left\{\begin{aligned}
& \hat{S} = \{ x_0^{b_2'} + x_2^{n_2} + x_3^{n_3} = 0 \} \subset \frac{1}{n_2 n_3} (-1, n_0 n_1 n_3, n_0 n_1 n_2) \\
& \mathcal{E}_1: \ x_0 = 0, \quad \hat{H}_2^{\text{red}}: \ x_2 = 0, \quad \hat{Y}: \ x_0^{b_2'} + x_2^{n_2} = 0.
\end{aligned}\right.
\end{equation}

\begin{figure}[ht]
\begin{tikzpicture}[scale=1]
\coordinate (A) at (0,1.35);
\coordinate (P1) at (2.5,1.5);
\coordinate (B) at (3.25,1.85);
\draw [name path=AP1] [thick] (A) to[out=25,in=150] (P1);
\draw [name path=P1B] [thick] (P1) to[out=50,in=180] (B);
\draw [name path=h0] [dashed] (0.5,1.25) node[below]{$\hat{H}_0$} -- (0.5,2.4);
\draw [name path=h1] [dashed] (1.5,1.25) node[below]{$\hat{H}_1$} -- (1.5,2.4);
\path [name intersections={of=AP1 and h0,by=Q0}];
\path [name intersections={of=AP1 and h1,by=Q1}];
\draw [dashed] ($(B)+(-1.3,0.2)$) to[out=330,in=120] (P1)
    to[out=70,in=200] ($(B)+(0.1,0.45)$) node[right]{$\hat{H}_2$ ($m$ components)};
\draw [dotted] ($(B)+(-1.3,0.4)$) node[above]{$\hat{Y}$} to[out=330,in=120] (P1)
    to[out=70,in=220] ($(B)+(-0.1,0.65)$);
\node [left] at (A) {$\mathcal{E}_1$};
\draw [fill] (Q0) node[above left]{$Q_0$} circle [radius=0.07];
\draw [fill] (Q1) node[above left]{$Q_1$} circle [radius=0.07];
\draw [fill] (P1) node[below]{$P_1$} circle [radius=0.07];
\end{tikzpicture}
\caption{First step of the $\Q$-resolution of $S$.}
\label{fig:1st-resolution-S}
\end{figure}
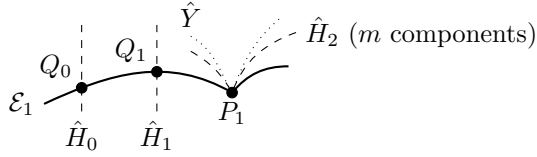

The points $Q_0$ and $Q_1$ already have $\mathbb{Q}$-normal crossings, so one does not need
to blow them up anymore. Consider the previous coordinates around $P_1$ and let $\pi_2$ be
the blow-up at $P_1$ with weights $w_2 = (1,\frac{b_2'}{n_2},\frac{b_2'}{n_3})$.
The exceptional divisor of $\pi_2$ is $E_2 = \mathbb{P}^2_{w_2}/G$ where $G$ is a cyclic
group of order $n_2 n_3$ acting diagonally as in~\eqref{eq:P1}. The exceptional divisor
$\mathcal{E}_2$ of the restriction $\varphi_2|_{\hat{S}}: \hat{S} \to \hat{S}$ is again
a rational irreducible curve containing $2+m$ cyclic quotient
singular points of $\hat{S}$, namely $Q_{12} = \mathcal{E}_1 \cap \mathcal{E}_2$,
$P_2 = \hat{Y} \cap \mathcal{E}_2$, and $m$ points $Q_{2j} \in \hat{H}_2 \cap \mathcal{E}_2$.

\begin{figure}[ht]
\begin{tikzpicture}[scale=1]
\coordinate (A) at (0,1.35);
\coordinate (B) at (3,1.35);
\coordinate (C) at ($(A)+(2,0)$);
\coordinate (D) at ($(B)+(2,0)$);
\draw [name path=AB] [thick] (A) to [out=25, in=150] (B);
\node [left] at (A) {$\mathcal{E}_1$};
\draw [name path=CD] [thick] (C) to [out=25, in=150] (D);
\node [right] at (D) {$\mathcal{E}_2$};
\draw [name path=h0] [dashed] (0.5,1) node[below]{$\hat{H}_0$} -- (0.5,2.4);
\draw [name path=h1] [dashed] (1.5,1) node[below]{$\hat{H}_1$} -- (1.5,2.4);
\draw [name path=h2a] [dashed] (3.4,1) node[below right]{$\hat{H}_2$} -- (3.4,2.4);
\draw [name path=h2b] [dashed] (4.05,1) -- (4.05,2.4);
\draw [name path=h3] [white, dashed] (4.75,1) -- (4.75,2.4);
\path [name intersections={of=AB and h0,by=Q0}];
\path [name intersections={of=AB and h1,by=Q1}];
\path [name intersections={of=CD and h2a,by=Q2a}];
\path [name intersections={of=CD and h2b,by=Q2b}];
\path [name intersections={of=CD and h3,by=P2}];
\path [name intersections={of=AB and CD,by=Q12}];
\draw [dotted] ($(P2)+(-0.45,0.7)$) to [out=330, in=100] (P2) to [out=80, in=210]
    ($(P2)+(0.45,0.7)$) node[right]{$\hat{Y}$};
\draw [fill] (Q0) node[above left]{$Q_0$} circle [radius=0.07];
\draw [fill] (Q1) node[above right]{$Q_1$} circle [radius=0.07];
\draw [fill] (Q12) node[below]{$Q_{12}$} circle [radius=0.07];
\draw [fill] (Q2a) node[above left]{$Q_{2j}$} circle [radius=0.07];
\draw [fill] (Q2b) circle [radius=0.07];
\draw [fill] (P2) node[below]{$P_2$} circle [radius=0.07];
\node [above right] at (Q2a) {$\stackrel{m}{\cdots}$};
\end{tikzpicture}
\caption{$\Q$-resolution of $S$.}
\label{fig:2nd-resolution-S}
\end{figure}
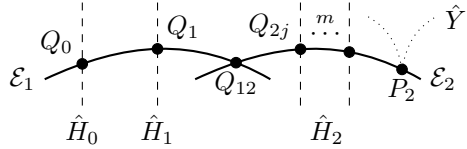

The composition $\varphi = \varphi_1 \circ \varphi_2 : \hat{S} \to S$ is a $\Q$-resolution of $S$
and the order of the groups at $Q_{12}$, $Q_{2j}$, and $P_2$ are $d$, $n_2$, and $\frac{n_3}{m}$, respectively.
Since the $\Q$-resolution graph is a tree and the exceptional divisors are isomorphic to $\mathbb{P}^1$
the link of $S$ is a rational homology sphere. In order to compute $\det S$ one needs to calculate the 
self-intersection numbers $\mathcal{E}_i^2 = -a_i$, $i=1,2$, which can be done by exploiting our information on 
the curve $Y$ in the partial resolution of $S$. First, note that the intersection of $\mathcal{E}_2$ with $\hat{Y}$ at 
$P_2$ is $m$. Second,
$$
\varphi^* Y = \hat{Y} + N_1 \mathcal{E}_1 + N_2 \mathcal{E}_2
$$
where $N_1 = n_0 n_1 n_2 n_3 = n$ and $N_2 = \frac{b_2'+n}{m}$. Since $\mathcal{E}_i \cdot \varphi^*Y = 0$, $i=1,2$, one obtains that
$a_1 = \frac{N_2}{N_1 d}$ and $a_2 = \frac{m+\frac{N_1}{d}}{N_2}$. Therefore the determinant of the intersection
matrix is given by
$$
\det (A) = \det \begin{pmatrix} - a_1 & \frac{1}{d} \\ \frac{1}{d} & - a_2 \end{pmatrix}
= \frac{m}{N_1 d}.
$$
By Lemma~\ref{lem:detS} one has
$$
\det S = \det(-A) n_0 n_1 d n_2^m \frac{n_3}{m} = n_2^{m-1}.
$$
Therefore by Corollary~\ref{cor:QHS} the link of $S$ is a integral homology sphere
if and only if $\det S = 1$, or equivalently, $m=1$ as claimed.
\end{proof}

\begin{rem}
If the exponents $n_i$'s are not pairwise coprime, then $\mathcal{E}_1 = \bigsqcup_j \mathcal{E}_{1j}$ has $n_{23} = \gcd(n_2,n_3)$ irreducible
components and $\mathcal{E}_2$ is irreducible. They have genus
$$
g(\mathcal{E}_{1j}) = \frac{1}{2} \left( \frac{n_{123}}{n_{23}} - 1 \right) \left( \frac{n_{023}}{n_{23}} - 1 \right)
\quad \text{and} \quad g(\mathcal{E}_2) = \frac{1}{2} (n_{23}-1)(m-1),
$$
where $m = \gcd(n_3,b)$ with $b = b_{20}n_1+b_{21}n_0$. The determinant of $S$ can be rewritten as
$$
\det S = \left( \frac{b}{m} \right)^{n_{23}-1} \left( \frac{N_1}{\alpha} \right)^{n_{123}-n_{23}}
\left( \frac{N_1}{\beta} \right)^{n_{023}-n_{23}} \left( \frac{n_2}{n_{23}} \right)^{m-1}
$$
where $N_1 = \lcm(n_0,n_1,n_2,n_3)$, $\alpha = \lcm(n_1,n_2,n_3)$, $\beta = \lcm(n_0,n_2,n_3)$,
and $n_{ijk} = \frac{n_i n_j n_k}{\lcm(n_i,n_j,n_k)}$. From here it can easily be shown that the link of $S$ is
an integral homology sphere if and only if $\gcd(n_i,n_j)=1$, $i\neq j$, and $m=1$.
The details are left to the reader.
\end{rem}

\bibliographystyle{amsplain}

\providecommand{\bysame}{\leavevmode\hbox to3em{\hrulefill}\thinspace}
\providecommand{\MR}{\relax\ifhmode\unskip\space\fi MR }
\providecommand{\MRhref}[2]{%
  \href{http://www.ams.org/mathscinet-getitem?mr=#1}{#2}
}
\providecommand{\href}[2]{#2}


\end{document}